\documentclass{amsart}
\usepackage{amssymb}
\usepackage{amsmath}
\usepackage{amsfonts}

\newtheorem{theorem}{Theorem}
\theoremstyle{plain}

\newtheorem{corollary}{Corollary}

\newtheorem{definition}{Definition}
\newtheorem{example}{Example}

\newtheorem{lemma}{Lemma}

\newtheorem{proposition}{Proposition}
\newtheorem{remark}{Remark}

\numberwithin{equation}{section}
\input{tcilatex}

\begin{document}
\title{ALMOST $\alpha $-PARACOSYMPLECTIC MANIFOLDS}
\author{I. K\"{u}peli Erken}
\address{Art and Science Faculty, Department of Mathematics, Uludag
University, 16059 Bursa, TURKEY}
\email{iremkupeli@uludag.edu.tr}
\author{ P. Dacko}
\email{piotrdacko@yahoo.com}
\author{C. Murathan }
\address{Art and Science Faculty, Department of Mathematics, Uludag
University, 16059 Bursa, TURKEY}
\email{cengiz@uludag.edu.tr}

\begin{abstract}
This paper is a complete study of almost $\alpha $-paracosmplectic
manifolds. We characterize almost $\alpha $-paracosmplectic manifolds which
have para Kaehler leaves. Main curvature identities which are fulfilled by
any almost $\alpha $-paracosmplectic manifold \ are found. We also proved
that $\xi $ is a harmonic vector field if and only if it is an eigen vector
field of the Ricci operator. We locally classify three dimensional almost $%
\alpha $-para-Kenmotsu manifolds satisfying a certain nullity condition. We
show that this condition is invariant under $D_{\gamma ,\beta }$-homothetic
deformation. Furthermore, we construct examples of almost $\alpha $%
-paracosmplectic manifolds satisfying generalized nullity conditions.
\end{abstract}

\maketitle


\section{I\textbf{ntroduction}}

The study of almost paracontact geometry was introduced by Kaneyuki and
Williams in \cite{kaneyuki1} and then it was continued by many other
authors. A systematic study of almost paracontact metric manifolds was
carried out in paper of Zamkovoy \cite{Za}. However such structures were
studied before \cite{RoscVanh}, \cite{BuchRosc}, \cite{BuchRosc2}. Note also 
\cite{Bejan}. These authors called such structures almost para-coHermitian.
The curvature identities for different classes of almost paracontact metric
manifolds were obtained e.g. in \cite{DACKO}, \cite{Welyczko}, \cite{Za}.

Considering the recent stage of the theory development there is an
impression that the geometers are focused on problems in almost paracontact
metric geometry which are seem to be created ad hoc, but in fact the source
for them lies in the Riemannian geometry of almost contact metric
structures. The basic reference for almost contact metric manifolds is a D.
E. Blair monograph \cite{Blair}. Recently appeared long awaited a survey
article \cite{CapMonNicYud} concernig almost cosymplectic manifolds as the
Blair's monograph deals mostly with contact metric manifolds.

Both almost contact metric and almost paracontact metric manifolds have
common roots in something we may call pre-cosymplectic structure which
simply is a pair of a 1-form usually denoted by $\eta $ and 2-form $\Phi $,
so $\eta \wedge \Phi ^{n}$ is a volume element. The characteristic (Reeb)
vector field $\xi $ is then defined by $i_{\xi }\eta =1$, $i_{\xi }\Phi =0$.
The Riemannian or pseudo-Riemannian geometry in this framework appears when
one is trying to introduce \emph{a compatible} structure which means a
metric or pseudo-metric $g$ and and affinor $\phi $ ((1,1)-tensor field),
such that $\Phi (X,Y)=g(X,\phi Y)$, and $\phi ^{2}=\epsilon (Id-\eta \otimes
\xi )$, where for $\epsilon =-1$ we have almost contact metric structure and
for $\epsilon =+1$ almost paracontact metric structure. The triple $(\phi
,\xi ,\eta )$ is then called almost contact structure or almost paracontact
structure, resp. For example: when $\eta $ is a contact form $d\eta =\Phi $
manifolds are called contact metric or paracontact metric, for both $\eta $, 
$\Phi $ closed, P. Libermann called such pair a cosymplectic structure, we
have almost cosymplectic manifolds or almost paracosymplectic manifolds.

The other possible point of view is to take an almost contact or almost
paracontact structures as a starting point and next to seek a compatible
metric or pseudo-metric.

Combining the assumption concernig the forms $\eta $, $\Phi $ and the
affinor $\phi $ we obtain several disjoint (rough) classes of manifolds.
Additionally within each of these classes are posed some assumptions
concernig the metric or pseudo-metric. Even if almost paracontact metric
manifolds were studied in the past it is recently when geometers discovered
many similarities between Riemannian and pseudo-Riemannian geometry of
almost contact metric and almost paracontact metric manifolds. Up to the
level when we can simply transliterate some properties.

Also this paper deals with the concept well-known in almost contact metric
geometry: manifolds with Reeb field belonging the the $\kappa $-nullity
ditribution and more general $(\kappa ,\mu $)-nullity or even $(\kappa ,\mu
,\nu )$-nullity distributions, here $\kappa $, $\mu $, $\nu $ are constants
or particular functions. Classifications are obtained for non-Sasakian
contact metric manifold, almost cosymplectic, almost $\alpha $-Kenmotsu and
almost $\alpha $-cosymplectic, \cite{Boeckx}, \cite{CarMol}, \cite{DacOl1}, 
\cite{DilPast2}, \cite{HAKAN}, \cite{PastSal}, \cite{Sal}.

The similar problems are now posed and studied for an almost paracontact
metric manifolds. However the situation is more difficult according to the
fact that occurs \textquotedblleft exceptional\textquotedblright\ manifolds,
that means manifolds without counterparts in the Riemannian case e.g. \cite%
{Dacko15}, \cite{MaMol}.

These \textquotedblleft exceptions\textquotedblright\ are often contraditcs
our intuitions. Also when thinking about tight relation between topology of
a manifold and its Riemannian geometry, particularly for closed manifolds,
from other hand pseudo-Riemannian metric are rather loosely related to the
manifold's topology we see that some problems can not be simply brought from
the almost contact metric geometry to almost paracontact.

Summarizing the contents of this paper, after the Preliminaries, where we
recall the definition of almost paracontact metric manifold, we introduce a
class of manifolds which contains both almost paracosymplectic and almost
para-Kenmotsu as well and we call these manifolds as almost $\alpha $%
-paracosymplectic, where $\alpha $ is a arbitrary function. However we prove
later on that in fact if dimension of the manifold is $\geqslant 5$, then
the 1-forms $d\alpha $ and $\eta $ are proportional.

There are basic objects for arbitrary almost paracontact metric manifold:
tensor fields $\mathcal{A}=-\nabla \xi $ and $h=\frac{1}{2}\mathcal{L}_{\xi
}\phi $. We study basic relations between them for the case of almost $%
\alpha $-paracosymplectic manifold. It is also established that $\xi $ is
geodesic and $\phi $ is $\xi $-parallel, $\nabla _{\xi }\phi =0$.

In the short auxiliary section we recall the concept of para-Kaeheler
manifolds we need to define a class of almost $\alpha$-paracosymplectic
manifolds with para-Kaehler leaves.

In the Sect. \textbf{5.} we characterize manifolds with para-Kaehler leaves:
an almost $\alpha $-paracosymplectic manifold has para-Kaehler leaves if and
only if 
\begin{equation*}
(\nabla _{X}\phi )Y=\alpha g(\phi X,Y)\xi +g(hX,Y)\xi -\alpha \eta (Y)\phi
X-\eta (Y)hX.
\end{equation*}

In the Sect. \textbf{6.} we determine $U(X,Y)=(\nabla _{\phi X}\phi )\phi
-(\nabla _{X}\phi )Y$ and other equivalent forms. One of the most important
object is a vector valued 2-form $\Omega $, defined as $\Omega
(X,Y)=R(X,Y)\xi $. We give its form in this section. Note that $\Omega $ is
more complicated for the case $\text{(}M)=3$ and $\alpha \neq const$. There
is a difference between 3- and higher-dimensional manifolds for $\alpha $
non-constant.

The Sect.\textbf{\ 7.} is particularly devoted to almost $\alpha $%
-paracosymplectic manifolds with $\alpha $ const. Such manifolds are also
known as almost $\alpha $-para-Kenmotsu manifolds and we follow this
terminolgy to emphasize that $\alpha =const$. In this section we obtain some
curvature identities for such manifolds. Also we provide more detailed study
of the Jacobi operator $lX=R(X,\xi )\xi $ and related objects. When manifold
has para-Kaehler leaves we measure the commutator $Q\phi -\phi Q$ with the
Ricci operator $Q$. Finally we notice that manifold with $h$ vanishing
everywhere has a simple local structure of a warped product $\mathbb{R}%
\times _{f}M$ of real line and almost para-Kaehler manifold.

When equip the tangent bundle of the manifold with a metric we can study the
problem of the \textquotedblleft harmonicity\textquotedblright\ of the
characteristic vector field $\xi $, where we consider $\xi $ as a map
between the manifold and its tangent bundle. For an almost $\alpha $%
-paracosymplectic manifolds $\xi $ is harmonic if ond only if it is an
eigenvector field of the Ricci operator, $Q\xi =f\xi $. This is proved in
the Sect. \textbf{8.}

In the Sect. \textbf{9.} it is proved that an almost $\alpha $%
-paracosymplectic manifold of dimension $\geqslant 5$ is locally conformal
to almost para-cosymplectic manifold and is locally $D_{1,\alpha }$%
-homothetic to almost para-Kenmotsu manifold near the points where $\alpha
\neq 0$.

In the Sect.\textbf{10.} there are considered so-called almost $\alpha $%
-para Kenmotsu $(\kappa ,\mu ,\nu )$-spaces. These manifolds are depicted by
the requirement that the form $R(X,Y)\xi $ is uniquely determined by the
respective Jacobi operator $lX=R(X,\xi )\xi $ in the way that $R(X,Y)\xi
=\eta (Y)lX-\eta (X)lY$. Then we assume that $l$ has very particular shape $%
l=\kappa \phi ^{2}+\mu h+\nu \phi h$, $\kappa $, $\mu $, $\nu $ are
constants or more generally functions however rather particular. The main
result in this section is that all these manifolds have para-Kahler leaves.

Finally in the last section we classify locally $3$-dimensional almost $%
\alpha $-para Kenmotsu manifolds studying possible canonical forms for the
tensor field $h$. As an application we describe the corresponding Ricci
operators. In this way it is discovered the connection between $3$-manifolds
with harmonic characteristic vector field and $(\kappa ,\mu ,\nu )$-spaces:
if $\xi $ is harmonic vector field then $M$ locally has a structure of $%
(\kappa ,\mu ,\nu )$-space, conversely for $3$-dimensional $(\kappa ,\mu
,\nu )$-space the characteristic vector field is harmonic.

\section{Preliminaries}

\label{preliminaries}

Let $M$ be a $(2n+1)$-dimensional differentiable manifold and $\phi $ is a $%
(1,1)$ tensor field, $\xi $ is a vector field and $\eta $ is a one-form on $%
M.$ Then $(\phi ,\xi ,\eta )$ is called an almost paracontact structure on $%
M $ if

(i)$\ \eta (\xi )=1,$ \ \ $\phi ^{2}=Id-\eta \otimes \xi ,~$

(ii) the tensor field $\phi $ induces an almost paracomplex structure on the
distribution $D=$ ker $\eta ,$ that is the eigendistributions $D^{\pm },$
corresponding to the eigenvalues $\pm $ $1,$ respectively have equal
dimensions, dim $D^{+}=$dim $D^{-}=$ $n.$ The manifold $M$ is said to be
almost paracontact manifold if it is endowed with an almost paracontact
structure \cite{Za}.

Let $M$ be an almost paracontact manifold. $M$ will be called an almost
paracontact metric manifold if it is additionally endowed with a
pseudo-Riemannian metric $g$ of a signature $(n+1,n)$, i.e.%
\begin{equation}
g(\phi X,\phi Y)=-g(X,Y)+\eta (X)\eta (Y).  \label{1}
\end{equation}

For such manifold, we additionally have 
\begin{equation}
\eta (X)=g(X,\xi ),\text{ }\phi (\xi )=0,\text{ }\eta \circ \phi =0.
\label{2}
\end{equation}

Moreover, we can define a skew-symmetric tensor field (a $2$-form) $\Phi $ by%
\begin{equation}
\Phi (X,Y)=g(\phi X,Y),  \label{3}
\end{equation}%
usually called a fundamental form corresponding to the structure. For an
almost $\alpha $-paracosymplectic manifold , there always exists an
orthogonal basis $\{X_{1},\ldots ,X_{n},Y_{1},\ldots ,Y_{n},\xi \}$ such
that $g(X_{i},X_{j})=\delta _{ij}$, $g(Y_{i},Y_{j})=-\delta _{ij}$ and $%
Y_{i}=\phi X_{i}$, for any $i,j\in \left\{ 1,\ldots ,n\right\} $. Such basis
is called a $\phi $-basis.

~On an almost paracontact manifold, one defines the $(2,1)$-tensor field $%
N^{(1)}$ by%
\begin{equation*}
N^{(1)}(X,Y)=\left[ \phi ,\phi \right] (X,Y)-2d\eta (X,Y)\xi ,
\end{equation*}%
where $\left[ \phi ,\phi \right] $ is the Nijenhuis torsion of $\phi $%
\begin{equation*}
\left[ \phi ,\phi \right] (X,Y)=\phi ^{2}\left[ X,Y\right] +\left[ \phi
X,\phi Y\right] -\phi \left[ \phi X,Y\right] -\phi \left[ X,\phi Y\right] .
\end{equation*}

If $N^{(1)}$ \ vanishes identically, then the almost paracontact manifold
(structure) is said to be normal \cite{Za}. The normality condition says
that the almost paracomplex structure $J$ defined on $M\times 
\mathbb{R}
$%
\begin{equation*}
J(X,\lambda \frac{d}{dt})=(\phi X+\lambda \xi ,\eta (X)\frac{d}{dt}),
\end{equation*}%
is integrable.

\section{Almost $\protect\alpha $-Paracosymplectic manifolds}

An almost paracontact metric manifold $M^{2n+1}$, with a\ structure $(\phi
,\xi ,\eta ,g)$ is said to be an almost $\alpha $-paracosymplectic manifold
if the form $\eta $ \ is closed and $d\Phi =2\alpha \eta \wedge \Phi ,$
where $\alpha $ may be a constant or a function on $M.$ Although $\alpha $
is arbitrary we will prove that if dimension $d=2n+1$ of $M$ is $\geqslant 5$%
, then $d\alpha =f\eta $ for a (smooth) function $f$.

For a particular choices of the function $\alpha $ we have the following
classes of manifolds

$\bullet $ almost $\alpha$-para-Kenmotsu manifolds%
\begin{equation*}
d\eta =0,\text{ \ \ }d\Phi =2\alpha\eta \wedge \Phi,\quad \alpha=const.,
\end{equation*}

$\bullet $ normal almost $\alpha$-para-Kenmotsu manifolds are called $\alpha$%
-para-Kenmotsu,

$\bullet $ almost paracosymplectic%
\begin{equation*}
d\eta =0,\text{ \ \ }d\Phi =0,
\end{equation*}%
quite similar normal almost paracosymplectic manifolds are paracosymplectic.

It is clear that almost $0$-para-Kenmotsu manifold is an almost
paracosymplectic manifold.

In what will follow we establish the fundamental properties of the
structure's tensor fields.

\begin{definition}
\label{AX} For an almost $\alpha $-paracosymplectic manifold, define the $%
(1,1)$-tensor field $\mathcal{A}$ by%
\begin{equation}
\mathcal{A}X=-\nabla _{X}\xi .  \label{4}
\end{equation}
\end{definition}

\begin{proposition}
\label{properties1} For an almost $\alpha $-paracosymplectic manifold $%
M^{2n+1}$, we have%
\begin{eqnarray}
i)\text{ }\mathcal{L}_{\xi }\eta &=&0,\text{ }ii)\text{\ }g(\mathcal{A}%
X,Y)=g(X,\mathcal{A}Y),\text{ }iii)\text{\ }\mathcal{A}\xi =0,  \notag \\
iv)\text{\ }\mathcal{L}_{\xi }\Phi &=&2\alpha \Phi ,\ v)(\mathcal{L}_{\xi
}g)(X,Y)=-2g(\mathcal{A}X,Y),\text{\ \ }  \notag \\
vi)\eta (\mathcal{A}X) &=&0,\ vii)\text{\ }d\alpha =f\eta \text{ if\ }%
n\geqslant 2  \label{5}
\end{eqnarray}%
where $\mathcal{L}$ indicates the operator of the Lie differentiation and $X$
is an arbitrary vector field on $M^{2n+1}$.
\end{proposition}

\begin{proof}
To prove $i)$ and $iv)$ we use the coboundry formula 
\begin{equation*}
\mathcal{L}_{\xi }\eta =d\circ i_{\xi }\eta +i_{\xi }\circ d\eta ,
\end{equation*}%
for the Lie derivative acting on skew-forms. We note that $i_{\xi }\eta =1$
and $d\eta =0.$ Similarly $(i_{\xi }\Phi )(X)=\Phi (\xi ,X)=0$ for an
arbitrary vector field, hence $i_{\xi }\Phi =0.$ Finally%
\begin{equation}
\text{\ }\mathcal{L}_{\xi }\Phi =i_{\xi }d\Phi =i_{\xi }(2\alpha \eta \wedge
\Phi )=2\alpha (i_{\xi }\eta \wedge \Phi -\eta \wedge i_{\xi }\Phi )=2\alpha
\Phi  \label{izeta}
\end{equation}%
Note $2d\eta (X,Y)=(\nabla _{X}\eta )(Y)-(\nabla _{Y}\eta )(X)=-g(\mathcal{A}%
X,Y)+g(X,\mathcal{A}Y),$ where the last equality follows from the definition
of $\mathcal{A}$. As $\eta $ is closed $\mathcal{A}$ is symmetric (or
self-adjoint), we completed the proof of $ii).$ Using the definition of Lie
differention and $\mathcal{A},$ we obtain 
\begin{eqnarray}
(\mathcal{L}_{\xi }g)(X,Y) &=&\xi g(X,Y)-g(\left[ \xi ,X\right] ,Y)-g(X,%
\left[ \xi ,Y\right] )  \label{LZETA1} \\
&=&-g(\mathcal{A}X,Y)-g(X,\mathcal{A}Y)=-2g(\mathcal{A}X,Y).  \label{LZETA2}
\end{eqnarray}%
(\ref{LZETA2}) implies $v).$ For $\xi $ is unit vector field we have for
arbitrary vector field $X$, $0=Xg(\xi ,\xi )=2g(\nabla _{X}\xi ,\xi )=-2\eta
(AX)=-2g(A\xi ,X)$ which yield $iii)$ and $vi)$. Finally to proof $vii)$ we
need the following
\end{proof}

\begin{lemma}
\label{podzero} Let $\omega $ be a 2-form on a manifold $\bar{M}$, $\text{dim%
}(\bar{M})=n\geqslant 4$ and $\omega $ has maximal rank at every point,
equivalently $\omega ^{\wedge \lbrack \frac{n}{2}]}$ is non-zero at every
point. If for a 1-form $\beta $ on $\bar{M}$, $\beta \wedge \omega =0$ at a
point $p\in \bar{M}$, then $\beta =0$ at $p$. Particularly $\beta $ vanishes
everywhere on $\bar{M}$ if $\beta \wedge \omega $ is everywhere zero.
\end{lemma}

\begin{proof}
Let $\beta \wedge \omega =0$ at $p$ and $\beta _{p}\neq 0$. Then there is a
vector $v$ at $p$, such that $\beta _{p}(v)=1$ and $i_{v}(\beta \wedge
\omega )_{p}=\omega _{p}-\beta _{p}\wedge \gamma _{p}$, $\gamma
_{p}=i_{v}\omega _{p}$. Hence $\omega _{p}=\beta _{p}\wedge \gamma _{p}$ and 
$\omega _{p}^{\wedge 2}=0$. In consequence as $[\frac{n}{2}]\geqslant 2$, $%
\omega _{p}^{\wedge \lbrack \frac{n}{2}]}=0$ which contradicts our
assumption that $\omega $ is of maximal rank.

Now we are going back to the proof of the part $vii)$. We put $\beta
=2\alpha \eta $. So $d\Phi =\beta \wedge \Phi $, applying exterior
differential to this equation and taking interior product with $i_{\xi }$ in
the result, we obtain $0=\gamma \wedge \Phi $ $(i_{\xi }\Phi =0)$
everywhere, $\gamma =i_{\xi }d\beta $. If $\text{dim}(M^{2n+1})\geqslant 5$ (%
$n\geqslant 2$) by the above Lemma $\gamma $ vanishes identically on $%
M^{2n+1}$. Notice $\gamma =i_{\xi }d\beta =2i_{\xi }(d\alpha \wedge \eta )$
as $d\eta =0$ and $0=(i_{\xi }d\alpha )\eta -d\alpha $, ($i_{\xi }\eta =1$).
Hence $d\alpha =f\eta $, $f=i_{\xi }d\alpha $.
\end{proof}

\begin{proposition}
\label{properties2} For an almost $\alpha $-paracosymplectic manifold, we
have%
\begin{equation}
\text{\ \ \ \ \ }\mathcal{A}\phi +\phi \mathcal{A}=-2\alpha \phi ,\text{ \ }%
\nabla _{\xi }\phi =0.  \label{6}
\end{equation}
\end{proposition}

\begin{proof}
$(\mathcal{L}_{\xi }\Phi )(X,Y)=\xi \Phi (X,Y)-\Phi (\left[ \xi ,X\right]
,Y)-\Phi (X,\left[ \xi ,Y\right] $ the definition of $\Phi $ follows%
\begin{eqnarray*}
(\mathcal{L}_{\xi }\Phi )(X,Y) &=&\xi g(\phi X,Y)-g(\phi \left[ \xi ,X\right]
,Y)-g(\phi X,\left[ \xi ,Y\right] ) \\
&=&g((\nabla _{\xi }\phi )X-\phi \mathcal{A}X-\mathcal{A}\phi X,Y).
\end{eqnarray*}%
We already know $\mathcal{L}_{\xi }\Phi =$ \ $2\alpha \Phi ,$ therefore
these both identities yield%
\begin{equation*}
2\alpha \phi X=(\nabla _{\xi }\phi )X-\phi \mathcal{A}X-\mathcal{A}\phi X.
\end{equation*}%
We have $\nabla _{\xi }\phi ^{2}=\nabla _{\xi }(Id-\eta \otimes \xi )=0$ for
both $\nabla _{\xi }\eta $ and $\nabla _{\xi }\xi $ vanish identically. From
other hand we have%
\begin{equation*}
(\nabla _{\xi }\phi ^{2})X=\phi (\nabla _{\xi }\phi )X+(\nabla _{\xi }\phi
)\phi X.
\end{equation*}%
Hence $\phi (\nabla _{\xi }\phi )X=-(\nabla _{\xi }\phi )\phi X$ and if $%
\phi X=X,$ that is $X$ is a field of eigenvectors corresponding to $+1$%
-eigenvalue $(\left[ +1\right] $-vector field$),$ then 
\begin{equation*}
2\alpha X=(\nabla _{\xi }\phi )X-\phi \mathcal{A}X-\mathcal{A}X,
\end{equation*}%
applying $\phi $ to the both hands we get%
\begin{equation*}
2\alpha X=\phi (\nabla _{\xi }\phi )X-\phi ^{2}\mathcal{A}X-\phi \mathcal{A}%
X=-(\nabla _{\xi }\phi )X-\mathcal{A}X-\phi \mathcal{A}X,
\end{equation*}%
and these both above identities follow $(\nabla _{\xi }\phi )X=-(\nabla
_{\xi }\phi )X=0.$ The same arguments prove $(\nabla _{\xi }\phi )X=0$ for $%
\left[ -1\right] $-vector field $\phi X=-X.$ Obviously $(\nabla _{\xi }\phi
)\xi =\nabla _{\xi }\phi \xi -\phi \nabla _{\xi }\xi =0.$ Therefore $\nabla
_{\xi }\phi =0$ identically as near each point there is a frame of vector
fields consisting only from $\xi $ and eigenvector fields of $\phi .$
\end{proof}

Let define $h=\frac{1}{2}\mathcal{L}_{\xi }\phi .$ In the following
proposition we establish some properties of the tensor field $h.$

\begin{proposition}
\label{h}For an almost $\alpha $-paracosymplectic manifold, we have the
following relations%
\begin{eqnarray}
\text{\ }g(hX,Y) &=&g(X,hY),  \label{6.1} \\
\text{ }h\circ \phi +\phi \circ h &=&0,\text{ }  \label{6.2} \\
\text{ }h\xi &=&0,  \label{6.3} \\
\nabla \xi &=&\alpha \phi ^{2}+\phi \circ h=-\mathcal{A}.\text{\ \ }
\label{7}
\end{eqnarray}
\end{proposition}

\begin{proof}
Similarly as in the Proposition \ref{properties2} we have%
\begin{equation}
(\mathcal{L}_{\xi }\phi ^{2})X=\phi (\mathcal{L}_{\xi }\phi )X+(\mathcal{L}%
_{\xi }\phi )\phi X=2\phi hX+2h\phi X.  \label{irem1}
\end{equation}%
and%
\begin{equation}
\mathcal{L}_{\xi }\phi ^{2}=-(\mathcal{L}_{\xi }\eta )\otimes \xi =0.
\label{irem2}
\end{equation}%
From (\ref{irem1}) and (\ref{irem2}) we get (\ref{6.2})$.$ By using the
formula $(\mathcal{L}_{\xi }\phi )X=\left[ \xi ,\phi X\right] -\phi \left[
\xi ,X\right] =\nabla _{\xi }\phi X-\nabla _{\phi X}\xi -\phi (\nabla _{\xi
}X-\nabla _{X}\xi )$ we obtain%
\begin{equation}
h=\frac{1}{2}(\mathcal{A}\phi -\phi \mathcal{A)}.  \label{hh}
\end{equation}%
The last formula and the properties of $\phi $ and $\mathcal{A(}$symmetry)
follow that $h$ is also a symmetric tensor field, $g(hX,Y)=g(X,hY).$
Moreover $h\xi =0$ and $\eta \circ h=0.$ Using (\ref{izeta}), (\ref{LZETA2})
and the following identity%
\begin{equation*}
(\mathcal{L}_{\xi }\Phi )(X,Y)=(\mathcal{L}_{\xi }g)(\phi X,Y)+g((\mathcal{L}%
_{\xi }\phi )X,Y),
\end{equation*}%
we obtain%
\begin{equation}
\alpha \phi =-\mathcal{A}\phi +h.  \label{alfai}
\end{equation}%
If we apply $\phi $ from the right to the (\ref{alfai}) and use the
anticommutative $h$ and $\phi ,$ we have%
\begin{equation*}
\alpha \phi ^{2}+\phi \circ h=-\mathcal{A=}\nabla \xi .
\end{equation*}
\end{proof}

\begin{corollary}
\label{trace}All the above Propositions imply the following formulas for the
traces%
\begin{eqnarray}
tr(\mathcal{A}\phi ) &=&tr(\phi \mathcal{A})=0,\ tr(h\phi )=tr(\phi h)=0, 
\notag \\
tr(\mathcal{A}) &=&-2\alpha n,\text{ \ }tr(h)=0.  \label{8}
\end{eqnarray}
\end{corollary}

\section{Para-Kaehler manifolds}

This is an auxiliary section. The general reference for the notions which
appear here is \cite{CFG}. We recall here basic concepts of a para-Hermitian
geometry. An even dimensional manifold $M^{2n\text{ }}$ endowed with a pair
an almost para-Hermitian structure $(J,<,>)$, where $J$ is an almost
para-complex structure and $<,>$ is a pseudo-Riemannian metric. These tensor
fields are subject of the following conditions%
\begin{equation*}
J^{2}=Id,\text{ \ \ \ }<JX,JY>=-<X,Y>,
\end{equation*}%
as it is common $X,Y$ denote vector fields. The manifold $M$ endowed with
this structure is called an almost para-Hermitian manifold. The almost
para-Hermitian manifold $M$ is para-Kaehler if the almost para-complex
structure $J$ is a covariant constant $\nabla J=0$, with respect to the
Levi-Civita connection. An almost para-complex structure is integrable if
and only if the Nijenhuis torsion of $J$ \ vanishes identically%
\begin{equation*}
N_{J}(X,Y)=J^{2}\left[ X,Y\right] +\left[ JX,JY\right] -J\left[ JX,Y\right]
-J\left[ X,JY\right] =0.
\end{equation*}%
An almost para-complex structure of a para-Kaehler manifold is always
integrable. In the terms of the local coordinates maps, integrability is
equivalent to the existence of a set of maps, covering the manifold, the
para-complex structure has constant coefficients in the local map
coordinates. If $p\in M$ is a point, then near $p$ we have coordinates $%
(x^{1},...x^{n},y^{1},...,y^{n}),$ the local components $J_{i}^{k}=const.$
are constants.

\section{Almost $\protect\alpha $-paracosymplectic manifolds with
para-Kaehler leaves}

The idea is to restrict further our consideration to the particular class of
manifolds. However this class of manifolds is wide enough to provide
interesting results and examples. In fact each $3$-dimensional manifold
belongs to this class. Let $M^{2n+1}=(M,\phi ,\xi ,\eta ,g)$ be an almost $%
\alpha $-paracosymplectic manifold. By the definition the form $\eta $ is
closed therefore a distribution $\mathcal{D}:\eta =0$ is completely
integrable. $\mathcal{D}$ defines a foliation $\mathcal{F}.$ Each leaf
carries an almost para-Kaehler structure $(J,<,>)$%
\begin{equation*}
J\bar{X}=\phi \bar{X},\text{ \ \ \ \ }\left\langle \bar{X},\bar{Y}%
\right\rangle =g(\bar{X},\bar{Y}),
\end{equation*}%
$\bar{X},\bar{Y}$ are vector fields tangent to the leaf. If this structure
is para-Kaehler, leaf is called a para-Kaehler leaf of the manifold $M.$

\begin{lemma}
\label{second}An almost $\alpha $-paracosymplectic manifold $M$ has
para-Kaehler leaves if and only if%
\begin{equation*}
(\nabla _{X}\phi )Y=g(\mathcal{A}X,\phi Y)\xi +\eta (Y)\phi \mathcal{A}X,%
\text{ \ \ \ }\mathcal{A}=-\nabla \xi .
\end{equation*}
\end{lemma}

\begin{proof}
Let $\mathcal{F}_{a}$ be a leaf passing through a point $a\in M$. The
characteristic vector field is a normal vector field to $\mathcal{F}_{a}$,
the restriction $\mathcal{A\mid }_{\mathcal{F}}\mathcal{=}-\nabla \xi \mid _{%
\mathcal{F}}$ is the Weingarten operator (the shape tensor). The Gauss
equation%
\begin{equation*}
\nabla _{\bar{X}}\bar{Y}=\bar{\nabla}_{\bar{X}}\bar{Y}+II(\bar{X},\bar{Y}%
)\xi ,
\end{equation*}%
yields%
\begin{eqnarray}
(\nabla _{\bar{X}}\phi )\bar{Y} &=&\nabla _{\bar{X}}\phi \bar{Y}-\phi \nabla
_{\bar{X}}\bar{Y}  \notag \\
&=&\bar{\nabla}_{_{\bar{X}}}J\bar{Y}+II(\bar{X},\phi \bar{Y})\xi  \label{J1}
\\
&&-\phi (\bar{\nabla}_{\bar{X}}\bar{Y}+II(\bar{X},\bar{Y})\xi )  \notag \\
&=&(\bar{\nabla}_{_{\bar{X}}}J)\bar{Y}+II(\bar{X},\phi \bar{Y})\xi =II(\bar{X%
},\phi \bar{Y})\xi  \notag
\end{eqnarray}%
here by assumption $\bar{\nabla}J=0$ identically, $II$ is the second
fundamental form of $\mathcal{F},$ $II(\bar{X},\bar{Y})=g(-\nabla _{\bar{X}%
}\xi ,\bar{Y}).$ The above identity implies $(\nabla _{X}\phi )Y=g(\mathcal{A%
}X,\phi Y)\xi $ for arbitrary vector fields on the manifold $M$ such that $%
\eta (X)=\eta (Y)=0.$ For arbitrary $X,Y$ we have a decomposition $X=(X-\eta
(X)\xi )+\eta (X)\xi .$ To finish the proof we need to remind that $\nabla
_{\xi }\phi =0$ and $(\nabla _{X}\phi )\xi =\phi \mathcal{A}X.$
\end{proof}

\begin{proposition}
Let $M^{2n+1}=(M,\phi ,\xi ,\eta ,g)$ be an almost $\alpha $%
-paracosymplectic manifold. Then the foliation $\mathcal{F}$, when $\alpha
=0 $ $($resp.$\alpha \neq 0)$,$\mathcal{F}$ is totally geodesic
(resp.totally umbilical$)$ if and only if $h=0.$
\end{proposition}

\begin{proof}
Using Gauss equation we have $II(\bar{X},\bar{Y})=g(\nabla _{\bar{X}}\bar{Y}%
,\xi )=-g(\bar{Y},\nabla _{\bar{X}}\xi )=-g(\bar{Y},\alpha \phi ^{2}\bar{X}%
+\phi h\bar{X})=-\alpha g(\bar{X},\bar{Y})-g(\bar{X},\phi h\bar{Y})$ for all 
$\bar{X},\bar{Y}\in \Gamma (D).$ This completes proof.
\end{proof}

\begin{proposition}
\label{alpha}An almost $\alpha $-paracosymplectic manifold $M$ has
para-Kaehler leaves if and only if%
\begin{equation}
(\nabla _{X}\phi )Y=\alpha g(\phi X,Y)\xi +g(hX,Y)\xi -\alpha \eta (Y)\phi
X-\eta (Y)hX  \label{MEHMET}
\end{equation}%
for $\alpha =0$ it is a formula known for almost paracosymplectic manifolds.
\end{proposition}

\begin{proof}
If we use the Lemma \ref{second} and the identity (\ref{7}), we have%
\begin{eqnarray*}
(\nabla _{X}\phi )Y &=&-g(\alpha \phi ^{2}X+\phi hX,\phi Y)\xi -\eta (Y)\phi
(\alpha \phi ^{2}X+\phi hX) \\
&=&-\alpha g(\phi ^{2}X,\phi Y)\xi -g(\phi hX,\phi Y)\xi -\alpha \eta
(Y)\phi X-\eta (Y)hX).
\end{eqnarray*}%
By the help of (\ref{1}) we get the requested equation.
\end{proof}

As a direct consequence we have the following

\begin{theorem}
Let $M^{2n+1}$ be an almost $\alpha $-para-Kenmotsu manifold with para
Kaehler leaves. Then $M^{2n+1}$ is a para-Kenmotsu ($\alpha =1$) manifold if
and only $\mathcal{A}=-\phi ^{2}.$
\end{theorem}

\begin{remark}
For a similar notion in contact metric geometry see e.g. \cite{OlK}, \cite%
{DilPast2}, \cite{HAKAN} and there are many other papers where this notion
appears explicitly or implicitly. Compare the references in \cite%
{CapMonNicYud}. We also note that in almost contact metric geometry there is
more general idea of when a manifold additionally carries a so-called
CR-structure. All almost contact metric manifolds with \emph{Kaehler} leaves
are also Levi-flat CR-manifolds.
\end{remark}

\section{Basic Structure and Curvature Identies}

\begin{lemma}
\label{2-form}For an almost $\alpha $-paracosymplectic manifold $(M,\phi
,\xi ,\eta ,g)$ with its fundamental $2$-form $\Phi $ the following
equations hold%
\begin{eqnarray}
\text{ }(\nabla _{X}\Phi )(Y,Z) &=&g((\nabla _{X}\phi )Y,Z),  \label{i} \\
\text{ }(\nabla _{X}\Phi )(Z,\phi Y)+(\nabla _{X}\Phi )(Y,\phi Z) &=&-\eta
(Y)g(\mathcal{A}X,Z)-\eta (Z)g(\mathcal{A}X,Y),  \label{ii} \\
(\nabla _{X}\Phi )(\phi Y,\phi Z)-(\nabla _{X}\Phi )(Y,Z) &=&\eta (Y)g(%
\mathcal{A}X,\phi Z)-\eta (Z)g(\mathcal{A}X,\phi Y),  \label{iii}
\end{eqnarray}%
where $\mathcal{A}=-\nabla \xi .$
\end{lemma}

\begin{proof}
The proof of (\ref{i}) is obvious. Differentiating the identity $\phi
^{2}=I-\eta \otimes \xi $ covariantly, we obtain%
\begin{equation}
(\nabla _{X}\phi )\phi Y+\phi (\nabla _{X}\phi )Y=g(Y,\mathcal{A}X)\xi +\eta
(Y)\mathcal{A}X.  \label{iv}
\end{equation}%
If we take the inner product with $Z,$ we obtain (\ref{ii})$.$ Replacing $Z$
by $\phi Z$ in (\ref{ii}), using the anti-symmetry of $\Phi $ and (\ref{i}),
we get (\ref{iii})$.$
\end{proof}

\begin{proposition}
\label{B(X,Y,Z)}For any almost $\alpha $-paracosymplectic manifold, we have%
\begin{equation}
(\nabla _{\phi X}\phi )\phi Y-(\nabla _{X}\phi )Y-\eta (Y)\mathcal{A}\phi
X-2\alpha (g(X,\phi Y)\xi +\eta (Y)\phi X)=0.  \label{B}
\end{equation}
\end{proposition}

\begin{proof}
Let us define $(0,3)$-tensor field $\mathcal{B}$ as follows%
\begin{equation*}
\mathcal{B}(X,Y,Z)=g((\nabla _{\phi X}\phi )\phi Y,Z)-g((\nabla _{X}\phi
)Y,Z)-\eta (Y)g(\mathcal{A}\phi X,Z)-2\alpha (g(X,\phi Y)\eta (Z)+\eta
(Y)g(\phi X,Z)).
\end{equation*}%
Antisymmetrizing $\mathcal{B}$ with respect to $X,Y$ \ we have%
\begin{eqnarray}
\mathcal{B}(X,Y,Z)-\mathcal{B}(Y,X,Z) &=&(\nabla _{\phi X}\Phi )(\phi
Y,Z)-(\nabla _{\phi Y}\Phi )(\phi X,Z)  \notag \\
&&-(\nabla _{X}\Phi )(Y,Z)+(\nabla _{Y}\Phi )(X,Z)  \notag \\
&&-\eta (Y)g(\mathcal{A}\phi X,Z)+\eta (X)g(\mathcal{A}\phi Y,Z)  \label{X-Y}
\\
&&-2\alpha ((g(X,\phi Y)-g(Y,\phi X))\eta (Z)  \notag \\
&&+\eta (Y)g(\phi X,Z)-\eta (X)g(\phi Y,Z)).  \notag
\end{eqnarray}%
Recalling the well known formula 
\begin{eqnarray*}
3d\Phi (X,Y,Z) &=&(\nabla _{X}\Phi )(Y,Z)+(\nabla _{Z}\Phi )(X,Y)+(\nabla
_{Y}\Phi )(Z,X) \\
&=&2\alpha (\eta (X)\Phi (Y,Z)+\eta (Z)\Phi (X,Y)+\eta (Y)\Phi (Z,X)).
\end{eqnarray*}%
and applying this in (\ref{X-Y}), we obtain%
\begin{eqnarray*}
\mathcal{B}(X,Y,Z)-\mathcal{B}(Y,X,Z) &=&-(\nabla _{Z}\Phi )(\phi X,\phi
Y)+(\nabla _{Z}\Phi )(X,Y) \\
&&-\eta (Y)g(\mathcal{A}\phi X,Z)+\eta (X)g(\mathcal{A}\phi Y,Z).
\end{eqnarray*}%
By (\ref{iii}), the right hand side of this equality vanishes identically,
so that $\mathcal{B}(X,Y,Z)-\mathcal{B}(Y,X,Z)=0,$ i.e. $\mathcal{B}$ is
symmetric with respect to $X,Y.$

Symmetrizing $\mathcal{B}$ with respect to $Y,Z,$ we find%
\begin{eqnarray*}
\mathcal{B}(X,Y,Z)+\mathcal{B}(X,Z,Y) &=&(\nabla _{\phi X}\Phi )(\phi
Y,Z)+(\nabla _{\phi X}\Phi )(\phi Z,Y) \\
&&-\eta (Y)g(\mathcal{A}\phi X,Z)-\eta (Z)g(\mathcal{A}\phi X,Y).
\end{eqnarray*}%
By the help of (\ref{ii}), we obtain $\mathcal{B}(X,Y,Z)+\mathcal{B}%
(X,Z,Y)=0,$ i.e. $\mathcal{B}$ is antisymmetric with respect to $Y,Z.$ The
tensor $\mathcal{B}$ having such symmetries must vanish identically, which
implies (\ref{B}).
\end{proof}

\begin{lemma}
\label{also have} For an almost $\alpha $-paracosymplectic manifold, we also
have%
\begin{eqnarray}
(\nabla _{\phi X}\phi )Y-(\nabla _{X}\phi )\phi Y+\eta (Y)\mathcal{A}%
X-2\alpha (g(X,Y)\xi -\eta (Y)X) &=&0,  \label{also1} \\
(\nabla _{\phi X}\phi )Y+\phi (\nabla _{X}\phi )Y-g(\mathcal{A}X,Y)\xi
-2\alpha (g(X,Y)\xi -\eta (Y)X) &=&0.  \label{also2}
\end{eqnarray}
\end{lemma}

\begin{proof}
Putting $\phi Y$ instead of $Y$ in (\ref{B}), we obtain%
\begin{equation}
(\nabla _{\phi X}\phi )Y-\eta (Y)(\nabla _{\phi X}\phi )\xi -(\nabla
_{X}\phi )\phi Y-2\alpha (g(X,Y)-\eta (Y)\eta (X)\xi )=0.  \label{also3}
\end{equation}%
Using (\ref{7}) and $(\nabla _{\phi X}\phi )\xi =\phi \mathcal{A}\phi X=-%
\mathcal{A}X-2\alpha\phi^2X$ in (\ref{also3}), we get (\ref{also1}).
Equation (\ref{also2}) comes from (\ref{iv}) and (\ref{also1}).
\end{proof}

Using (\ref{also2}), one can easily get following

\begin{proposition}
\label{BEN}For any almost $\alpha $-paracosymplectic manifold, we have%
\begin{equation}
\phi (\nabla _{\phi X}\phi )Y+(\nabla _{X}\phi )Y=-2\alpha \eta (Y)\phi
X+g(\alpha \phi X+hX,Y)\xi .  \label{original}
\end{equation}
\end{proposition}

\begin{theorem}
\label{R0}Let ($M^{2n+1},\phi ,\xi ,\eta ,g)$ be an almost $\alpha $%
-paracosymplectic manifold. Then, for any $X,Y\in \chi (M^{2n+1}),$%
\begin{equation}
\begin{array}{rcl}
R(X,Y)\xi & = & d\alpha(X)(Y-\eta(Y)\xi)-d\alpha(Y)(X-\eta(X)\xi)+ \alpha
\eta (X)(\alpha Y+\phi hY) \\[+4pt] 
&  & - \,\alpha \eta (Y)(\alpha X+\phi hX)+(\nabla _{X}\phi h)Y-(\nabla
_{Y}\phi h)X. \label{b30}%
\end{array}%
\end{equation}
\end{theorem}

\begin{proof}
We have the Ricci identity for the alteration the second covariant
derivative $\nabla_{X,Y}\xi-\nabla_{Y,X}\xi=R(X,Y)\xi$. We notice that $%
\nabla_{X,Y}\xi=-(\nabla_X\mathcal{A})Y$. Now if we substitute $\mathcal{A}$
according to (\ref{7}) and applying the covariant derivative to the all
summands in the result we obtain 
\begin{equation}
\nabla_{X,Y}\xi = d\alpha(X)(Y-\eta(Y)\xi)+\alpha\eta(X)(\alpha Y+\phi hY)
+(\nabla_X\phi h)Y.
\end{equation}
\end{proof}

The identity for the curvature $R(X,Y)\xi$ greatly simplifies if $\text{dim}%
(M) \geqslant 5$ according to the Proposition \textbf{\ref{properties1}}($%
vii)$.

\begin{corollary}
For an almost $\alpha$-paracosymplectic manifold $M^{2n+1}$, $n\geqslant 2$

\begin{equation}  \label{rxyxi2}
\begin{array}{rcl}
R(X,Y)\xi & = & (f+\alpha^2)(\eta(X)Y-\eta(Y)X)+ \alpha(\eta(X)\phi
hY-\eta(Y)\phi hX) \\[+4pt] 
&  & +\,(\nabla_X\phi h)Y-(\nabla_Y\phi h)X,%
\end{array}%
\end{equation}
where $f = i_\xi d\alpha$.
\end{corollary}

\section{Almost $\protect\alpha$-para-Kenmotsu manifolds}

In this section we study particularly almost $\alpha$-para-Kenmotsu
manifolds if it is not otherwise stated.

\begin{theorem}
\label{R}Let ($M^{2n+1},\phi ,\xi ,\eta ,g)$ be an almost $\alpha $%
-para-Kenmotsu manifold. Then, for any $X,Y\in \chi (M^{2n+1}),$%
\begin{equation}
R(X,Y)\xi = \alpha \eta (X)(\alpha Y+\phi hY) - \,\alpha \eta (Y)(\alpha
X+\phi hX)+(\nabla _{X}\phi h)Y-(\nabla _{Y}\phi h)X.  \label{b3}
\end{equation}
\end{theorem}

\begin{proof}
It is direct consequence of the Theorem \textbf{\ref{R0}} for $\alpha$ is a
constant.
\end{proof}

\begin{theorem}
\label{R2}Let ($M^{2n+1},\phi ,\xi ,\eta ,g)$ be an almost $\alpha $%
-para-Kenmotsu manifold. Then, for any $X\in \chi (M^{2n+1})$ we have%
\begin{eqnarray}
R(\xi ,X)\xi &=&\alpha ^{2}\phi ^{2}X+2\alpha \phi hX-h^{2}X+\phi (\nabla
_{\xi }h)X,  \label{iremm} \\
(\nabla _{\xi }h)X &=&-\alpha ^{2}\phi X-2\alpha hX+\phi h^{2}X-\phi R(X,\xi
)\xi ,  \label{iremmm} \\
\frac{1}{2}(R(\xi ,X)\xi +\phi R(\xi ,\phi X)\xi ) &=&\alpha ^{2}\phi
^{2}X-h^{2}X.  \label{iremmmm} \\
S(X,\xi ) &=&-2n\alpha ^{2}\eta (X)+g(div(\phi h),X)  \label{SZETA} \\
S(\xi ,\xi ) &=&-2n\alpha ^{2}+trh^{2}.  \label{SZETAZETA}
\end{eqnarray}
\end{theorem}

\begin{proof}
If we replace $X$ by $\xi $ and $Y$ by $X$ in (\ref{b3}) and use (\ref{7})
we obtain (\ref{iremm}). For the proof of (\ref{iremmm}), we apply the
tensor field $\phi $ both sides of the (\ref{iremm}) and recall $\nabla
_{\xi }\phi =0$. Hence we have%
\begin{equation*}
-\phi R(X,\xi )\xi =\alpha ^{2}\phi X+2\alpha hX-\phi h^{2}X+(\nabla _{\xi
}h)X-g((\nabla _{\xi }h)X,\xi )\xi
\end{equation*}%
Replacing $X$ by $\phi X$ in (\ref{iremm}) we get%
\begin{equation*}
R(\xi ,\phi X)\xi =\alpha ^{2}\phi ^{3}X+2\alpha \phi h\phi X-h^{2}\phi
X+\phi (\nabla _{\xi }h)\phi X.
\end{equation*}%
If we apply $\phi $ to the last equation we have%
\begin{equation}
\phi R(\xi ,\phi X)\xi =\alpha ^{2}\phi ^{2}X+2\alpha h\phi X-h^{2}X+(\nabla
_{\xi }h)\phi X.  \label{irem4}
\end{equation}%
One can easily show that $\phi (\nabla _{\xi }h)X=-(\nabla _{\xi }h)\phi X.$
Combining (\ref{iremm}) with (\ref{irem4}) we get (\ref{iremmmm}).

Taking into acount $\phi $-basis and (\ref{b3}), Ricci curvature $S(X,\xi )$
can be given by%
\begin{eqnarray}
S(X,\xi ) &=&\dsum\limits_{i=1}^{n}[g(R(e_{i},X)\xi ,e_{i})-g(R(\phi
e_{i},X)\xi ,\phi e_{i})]  \label{osman} \\
&=&-2n\alpha ^{2}\eta (X)-\dsum\limits_{i=1}^{n}(g((\nabla _{X}\phi
h)e_{i},e_{i})-g((\nabla _{X}\phi h)\phi e_{i},\phi e_{i}))  \notag \\
&&+\dsum\limits_{i=1}^{n}(g((\nabla _{e_{i}}\phi h)X,e_{i})-g((\nabla _{\phi
e}\phi h)X,\phi e_{i}))  \notag
\end{eqnarray}%
After some calculations we have%
\begin{equation*}
\dsum\limits_{i=1}^{n}(g((\nabla _{X}\phi h)e_{i},e_{i})-g((\nabla _{X}\phi
h)\phi e_{i},\phi e_{i}))=0,
\end{equation*}%
\begin{equation*}
\dsum\limits_{i=1}^{n}(g((\nabla _{e_{i}}\phi h)X,e_{i})-g((\nabla _{\phi
e_{i}}\phi h)X,\phi e_{i}))=g(div(\phi h),X).
\end{equation*}%
Using the last two equations in (\ref{osman}) we obtain%
\begin{equation*}
S(X,\xi )=-2n\alpha ^{2}\eta (X)+g(div(\phi h),X).
\end{equation*}%
By direct calculation, we find%
\begin{equation*}
S(\xi ,\xi )=-2n\alpha ^{2}+trh^{2}.
\end{equation*}
\end{proof}

\begin{proposition}
\label{R3}Let ($M^{2n+1},\phi ,\xi ,\eta ,g)$ be an almost $\alpha $%
-para-Kenmotsu manifold. Then, for any $X,Y,Z\in \chi (M^{2n+1})$ we have%
\begin{eqnarray}
&&g(R(\xi ,X)Y,Z)+g(R(\xi ,X)\phi Y,\phi Z)-g(R(\xi ,\phi X)\phi
Y,Z)-g(R(\xi ,\phi X)Y,\phi Z)  \notag \\
&=&2(\nabla _{hX}\Phi )(Y,Z)+2\alpha ^{2}\eta (Y)g(X,Z)-2\alpha ^{2}\eta
(Z)g(X,Y)  \notag \\
&&-2\alpha \eta (Z)g(\phi hX,Y)+2\alpha \eta (Y)g(\phi hX,Z).  \label{irem5}
\end{eqnarray}
\end{proposition}

\begin{proof}
The symmetries of the curvature tensor give $g(R(\xi ,X)Y,Z)=g(X,R(Y,Z)\xi )$
and then, using (\ref{b3}), the left hand side can be written as%
\begin{equation}
2\alpha ^{2}\eta (Y)g(X,Z)-2\alpha ^{2}\eta (Z)g(X,Y)+\mathcal{F}(X,Y,Z)-%
\mathcal{F}(X,Z,Y),  \label{irem6}
\end{equation}%
where%
\begin{eqnarray*}
\mathcal{F}(X,Y,Z) &=&g(X,(\nabla _{Y}\phi h)Z+\phi (\nabla _{Y}\phi h)\phi
Z) \\
&&+g(X,(\nabla _{\phi Y}\phi h)\phi Z)-g(\phi X,(\nabla _{\phi Y}\phi h)Z).
\end{eqnarray*}%
By direct computations we have%
\begin{equation}
\phi (\nabla _{Y}\phi h)\phi Z+(\nabla _{Y}\phi h)Z=(\nabla _{Y}\phi
)hZ-h(\nabla _{Y}\phi )Z,  \label{irem7}
\end{equation}

and%
\begin{eqnarray}
g(X,(\nabla _{\phi Y}\phi h)\phi Z)-g(\phi X,(\nabla _{\phi Y}\phi h)Z)
&=&-g(\phi X,\phi ((\nabla _{\phi Y}\phi h)\phi Z))  \notag \\
&&+\eta (X)\eta ((\nabla _{\phi Y}\phi h)\phi Z)-g(\phi X,(\nabla _{\phi
Y}\phi h)Z).  \label{irem8}
\end{eqnarray}%
Using (\ref{irem7}), (\ref{irem8}), (\ref{original}) and the equality $\eta
((\nabla _{\phi Y}\phi h)\phi Z)=g(hZ,\alpha \phi Y-hY),$we obtain%
\begin{equation}
\mathcal{F}(X,Y,Z)=-2g(hX,(\nabla _{Y}\phi )Z)-2\alpha \eta (Z)g(h\phi
Y,X)+2\alpha \eta (X)g(h\phi Y,Z).  \label{irem9}
\end{equation}%
Using (\ref{irem9}) in (\ref{irem6}), the required formula $\mathcal{\sigma }%
_{Y,Z,hX}(\nabla _{Y}\Phi )(Z,hX)=d\Phi (Y,Z,hX)$ and $d\Phi =2\alpha \eta
\wedge \Phi $ $.$
\end{proof}

\begin{theorem}
\label{Q}Let $(M^{2n+1},\phi ,\xi ,\eta ,g)$ be an almost $\alpha $%
-para-Kenmotsu manifold with para-Kaehler leaves. Then the following
identity holds%
\begin{equation}
Q\phi -\phi Q=l\phi -\phi l-4\alpha (1-n)h-\eta \otimes \phi Q+(\eta \circ
Q\phi )\otimes \xi ,  \label{Q1}
\end{equation}%
where $l$ denotes the Jacobi operator, defined by $lX=R(X,\xi )\xi .$
\end{theorem}

\begin{proof}
We recall the formula (\ref{b3}) 
\begin{equation}
R(X,Y)\xi =\alpha \eta (X)(\alpha Y+\phi hY)-\alpha \eta (Y)(\alpha X+\phi
hX)+(\nabla _{X}\phi h)Y-(\nabla _{Y}\phi h)X.  \notag
\end{equation}%
On the other hand 
\begin{equation}
(\nabla _{X}\phi h)Y=(\nabla _{X}\phi )hY+\phi ((\nabla _{X}h)Y).
\label{PARAKEN2}
\end{equation}%
Using (\ref{MEHMET}) and (\ref{PARAKEN2}), we obtain 
\begin{eqnarray}
R(X,Y)\xi &=&\alpha \eta (X)(\alpha Y+\phi hY)-\alpha \eta (Y)(\alpha X+\phi
hX)  \label{PARAKEN3} \\
&&+\phi ((\nabla _{X}h)Y-(\nabla _{Y}h)X)+(\nabla _{X}\phi )hY-(\nabla
_{Y}\phi )hX.  \notag
\end{eqnarray}%
By assumption $M^{2n+1}$ has para-Kaehler leaves thus by (\ref{MEHMET}) $%
(\nabla _{X}\phi )hY=\alpha g(\phi X,hY)\xi +g(hX,hY)\xi $ in consequence,
as $h\phi $ is symmetric, $(\nabla _{X}\phi )hY-(\nabla _{Y}\phi )hX$
vanishes identically. Since $h$ is a symmetric operator we easily get 
\begin{equation}
g((\nabla _{X}h)Y-(\nabla _{Y}h)X,\xi )=g((\nabla _{X}h)\xi ,Y)-g((\nabla
_{Y}h)\xi ,X).  \label{PARAKEN4}
\end{equation}%
Using the formulas (\ref{7}), $h\xi =0$ and $\phi h+h\phi =0$ in (\ref%
{PARAKEN4}) we find 
\begin{equation}
g((\nabla _{X}h)Y-(\nabla _{Y}h)X,\xi )=2g(\phi h^{2}X,Y).  \label{PARAKEN5}
\end{equation}%
By applying $\phi $ to (\ref{PARAKEN3}) and using $\phi ^{2}=I-\eta \otimes
\xi $ and (\ref{PARAKEN5})\ we obtain%
\begin{eqnarray}
(\nabla _{X}h)Y-(\nabla _{Y}h)X &=&\phi R(X,Y)\xi +2g(\phi h^{2}X,Y)\xi
\label{Q2} \\
&&-\alpha ^{2}(\eta (X)\phi Y-\eta (Y)\phi X)-\alpha (\eta (X)hY-\eta (Y)hX).
\notag
\end{eqnarray}%
Now we suppose that $P$ is a fixed point of $M$ and $X,Y,Z$ are vector
fields such that $(\nabla X)_{P}=$ $(\nabla Y)_{P}=$ $(\nabla Z)_{P}=0.$ The
Ricci identity for $\phi $%
\begin{equation*}
R(X,Y)\phi Z-\phi R(X,Y)Z=(\nabla _{X}\nabla _{Y}\phi )Z-(\nabla _{Y}\nabla
_{X}\phi )Z-(\nabla _{\left[ X,Y\right] }\phi )Z,
\end{equation*}%
at the point $P$, reduces to the form%
\begin{equation*}
R(X,Y)\phi Z-\phi R(X,Y)Z=\nabla _{X}(\nabla _{Y}\phi )Z-\nabla _{Y}(\nabla
_{X}\phi )Z.
\end{equation*}%
Due to our assumption that $M^{2n+1}$ has para-Kaehler leaves from (\ref%
{MEHMET}) we obtain at $P$ By virtue of the integrability condition we have,
at $P$,%
\begin{eqnarray}
R(X,Y)\phi Z-\phi R(X,Y)Z &=&\nabla _{X}(\nabla _{Y}\phi )Z-\nabla
_{Y}(\nabla _{X}\phi )Z  \notag \\
&=&\alpha \left( g((\nabla _{X}\phi )Y-(\nabla _{Y}\phi )X,Z)\xi -\eta
(Z)((\nabla _{X}\phi )Y-(\nabla _{Y}\phi )X)\right)  \notag \\
&&+g((\nabla _{X}h)Y-(\nabla _{Y}h)X,Z)\xi -\eta (Z)((\nabla _{X}h)Y-(\nabla
_{Y}h)X))  \notag \\
&&+g(\alpha \phi Y+hY,Z)(\alpha \phi ^{2}X+\phi hX)-g(\alpha \phi
X+hX,Z)(\alpha \phi ^{2}Y+\phi hY)  \notag \\
&&-g(Z,\alpha \phi ^{2}X+\phi hX)(\alpha \phi Y+hY)+g(Z,\alpha \phi
^{2}Y+\phi hY)(\alpha \phi X+hX).  \label{Q3}
\end{eqnarray}%
Using (\ref{MEHMET}) and (\ref{Q2}) in (\ref{Q3}) we find%
\begin{eqnarray}
R(X,Y)\phi Z-\phi R(X,Y)Z &=&g(\phi R(X,Y)\xi ,Z)\xi -\eta (Z)\phi R(X,Y)\xi
\notag \\
&&+g(\alpha \phi Y+hY,Z)(\alpha \phi ^{2}X+\phi hX)-g(\alpha \phi
X+hX,Z)(\alpha \phi ^{2}Y+\phi hY)  \notag \\
&&-g(Z,\alpha \phi ^{2}X+\phi hX)(\alpha \phi Y+hY)+g(Z,\alpha \phi
^{2}Y+\phi hY)(\alpha \phi X+hX).  \label{Q4}
\end{eqnarray}%
Using (\ref{1}) and the curvature tensor properties we get%
\begin{equation}
g(\phi R(\phi X,\phi Y)Z,\phi W)=-g(R(Z,W)\phi X,\phi Y)+\eta (R(\phi X,\phi
Y)Z)\eta (W).  \label{Q5}
\end{equation}%
Then by (\ref{Q4}) and (\ref{Q5}) we obtain%
\begin{eqnarray}
g(\phi R(\phi X,\phi Y)Z,\phi W) &=&-g(\phi R(Z,W)X,\phi Y)+\eta (R(\phi
X,\phi Y)Z)\eta (W)  \notag \\
&&+\eta (X)g(\phi R(Z,W)\xi ,\phi Y)  \notag \\
&&-g(\alpha \phi W+hW,X)(g(\alpha \phi ^{2}Z,\phi Y)+g(\phi hZ,\phi Y)) 
\notag \\
&&+g(\alpha \phi Z+hZ,X)(g(\alpha \phi ^{2}W,\phi Y)+g(\phi hW,\phi Y)) 
\notag \\
&&+g(X,\alpha \phi ^{2}Z+\phi hZ)((g(\alpha \phi W,\phi Y)+g(hW,\phi Y)) 
\notag \\
&&-g(X,\alpha \phi ^{2}W+\phi hW)((g(\alpha \phi Z,\phi Y)+g(hZ,\phi Y)).
\label{Q6}
\end{eqnarray}%
Replacing in (\ref{Q4}) $X,Y$ by $\phi X,\phi Y$ respectively, and taking
the inner product with $\phi W,$ we \ get%
\begin{eqnarray}
g(R(\phi X,\phi Y)\phi Z,\phi W)-g(\phi R(\phi X,\phi Y)Z,\phi W) &=&-\eta
(Z)g(\phi R(\phi X,\phi Y)\xi ,\phi W)  \notag \\
&&+g(\alpha \phi ^{2}Y+h\phi Y,Z)g(\alpha \phi ^{3}X+\phi h\phi X,\phi W) 
\notag \\
&&-g(\alpha \phi ^{2}X+h\phi X,Z)g(\alpha \phi ^{3}Y+\phi h\phi Y,\phi W) 
\notag \\
&&-g(Z,\alpha \phi ^{3}X+\phi h\phi X)g(\alpha \phi ^{2}Y+h\phi Y,\phi W) 
\notag \\
&&+g(Z,\alpha \phi ^{3}Y+\phi h\phi Y)g(\alpha \phi ^{2}X+h\phi X,\phi W).
\label{Q7}
\end{eqnarray}%
Comparing (\ref{Q6}) to (\ref{Q7}) we get by direct computation%
\begin{eqnarray}
g(R(\phi X,\phi Y)\phi Z,\phi W) &=&g(R(Z,W)X,Y)-\eta (R(Z,W)X)\eta (Y) 
\notag \\
&&-\eta (X)g(R(Z,W)\xi ,Y)+\eta (R(\phi X,\phi Y)Z)\eta (W)  \notag \\
&&-\eta (Z)g(\phi R(\phi X,\phi Y)\xi ,\phi W)  \notag \\
&&-2\alpha g(X,Z)g(Y,\phi hW)+2\alpha \eta (X)\eta (Z)g(\phi hW,Y)  \notag \\
&&+2\alpha g(Y,Z)g(X,\phi hW)-2\alpha \eta (Y)\eta (Z)g(\phi hW,X)  \notag \\
&&+2\alpha g(X,W)g(Y,\phi hZ)-2\alpha \eta (X)\eta (W)g(\phi hZ,Y)  \notag \\
&&+2\alpha \eta (Y)\eta (W)g(X,\phi hZ)-2\alpha g(Y,W)g(\phi hZ,X).
\label{Q8}
\end{eqnarray}%
Let $\left\{ e_{i},\phi e_{i},\xi \right\} ,i\in \left\{ 1,...n\right\} ,$
be a local $\phi $-basis. Setting $Y=Z=e_{i}$ in (\ref{Q8}), we have%
\begin{eqnarray}
\dsum\limits_{i=1}^{n}g(R(\phi X,\phi e_{i})\phi e_{i},\phi W)
&=&\dsum\limits_{i=1}^{n}(g(R(e_{i},W)X,e_{i})-\eta (X)g(R(e_{i},W)\xi
,e_{i})+\eta (R(\phi X,\phi e_{i})e_{i})\eta (W)  \notag \\
&&-2\alpha g(X,e_{i})g(e_{i},\phi hW)+2\alpha g(e_{i},e_{i})g(\phi hW,X) 
\notag \\
&&+2\alpha g(X,W)g(e_{i},\phi he_{i})-2\alpha \eta (X)\eta (W)g(\phi
he_{i},e_{i})  \notag \\
&&-2\alpha g(e_{i},W)g(\phi he_{i},X)).  \label{Q9}
\end{eqnarray}%
On the other hand, putting $Y=Z=\phi e_{i}$ in (\ref{Q8}), we get%
\begin{eqnarray}
\dsum\limits_{i=1}^{n}g(R(\phi X,e_{i})e_{i},\phi W)
&=&\dsum\limits_{i=1}^{n}(g(R(\phi e_{i},W)X,\phi e_{i})-\eta (X)g(R(\phi
e_{i},W)\xi ,\phi e_{i})+\eta (R(\phi X,e_{i})\phi e_{i})\eta (W)  \notag \\
&&-2\alpha g(X,\phi e_{i})g(\phi e_{i},\phi hW)+2\alpha g(\phi e_{i},\phi
e_{i})g(\phi hW,X)  \notag \\
&&+2\alpha g(X,W)g(\phi e_{i},\phi h\phi e_{i})-2\alpha \eta (X)\eta
(W)g(\phi h\phi e_{i},\phi e_{i})  \notag \\
&&-2\alpha g(\phi e_{i},W)g(\phi h\phi e_{i},X)).  \label{Q10}
\end{eqnarray}%
Using the definition of the Ricci operator, (\ref{Q9}) and (\ref{Q10}), one
can easily get%
\begin{eqnarray}
-\phi Q\phi X+\phi l\phi X+QX-lX &=&\eta (X)Q\xi +4\alpha (1-n)\phi hX 
\notag \\
&&+\dsum\limits_{i=1}^{n}(g(R(\phi X,e_{i})\phi e_{i},\xi )-g(R(\phi X,\phi
e_{i})e_{i},\xi ))\xi .  \label{Q11}
\end{eqnarray}%
Finally, applying $\phi $ to (\ref{Q11}) and using $\phi ^{2}=I-\eta \otimes
\xi ,$ we obtain the requested equation.
\end{proof}

\begin{theorem}
Let $M^{2n+1}$ be an almost $\alpha $-para-Kenmotsu manifold of constant
sectional curvature $c$. Then $c=-\alpha ^{2}$ and $h^{2}=0.$
\end{theorem}

\begin{proof}
If an almost $\alpha $-para-Kenmotsu manifold of constant sectional
curvature $c$ then 
\begin{equation}
R(\xi ,X)\xi =c(\eta (X)\xi -X)=\phi R(\xi ,\phi X)\xi .  \label{fuat}
\end{equation}%
for any $X$ $\in \Gamma (M)$. Using this relation in (\ref{iremmmm}) we have%
\begin{equation}
h^{2}X=(\alpha ^{2}+c)\phi ^{2}X  \label{Goknur}
\end{equation}
Differentiating (\ref{Goknur}) with respect to $\xi $ and using $\nabla
_{\xi }\phi =0$ we find $\nabla _{\xi }h^{2}=0$ which implies 
\begin{equation*}
(\nabla _{\xi }h)\circ h+h\circ (\nabla _{\xi }h)=0.
\end{equation*}%
Applying $\nabla _{\xi }$ to the above equation and using (\ref{iremmm}), we
get $(\nabla _{\xi }h)^{2}=0.$ Since $\nabla _{\xi }h$ is symmetric operator
one easily have%
\begin{equation}
0=g((\nabla _{\xi }h)^{2}X,Y)=g((\nabla _{\xi }h)X,(\nabla _{\xi }h)Y).
\label{Cemal}
\end{equation}%
By virtue of (\ref{Goknur}), (\ref{fuat}) and (\ref{iremmm}) we find%
\begin{equation*}
(\nabla _{\xi }h)X=-2\alpha hX
\end{equation*}%
Hence (\ref{Cemal}) is reduce to $4\alpha ^{2}g(h^{2}X,Y)=4\alpha
^{2}(\alpha ^{2}+c)g(\phi X,\phi Y)=0.$

Because of $\alpha \neq 0$, we obtain $c=-\alpha ^{2}$ and $h^{2}=0.$
\end{proof}

The proof of the following theorem is exactly same with almost Kenmotsu
manifolds \cite{DILEO}, therefore we omit their proofs.

\begin{theorem}
Let $M^{2n+1}$ be an almost $\alpha $- para Kenmotsu manifold with $h=0.$
Then $M^{2n+1}$ is a locally warped product $M_{1}\times _{f^{2}}M_{2},$%
where $M_{2}$ is an almost para Kaehler manifold, $M_{1}$ is an open
interval with coordinate $t$, and $f^{2}=we^{2t}$ for some positive constant.
\end{theorem}

\begin{remark}
Almost Kenmotsu manifolds in almost contact metric geometry appeared in \cite%
{DILEO}, \cite{KimPak} and \cite{Olszak}. These manifolds were extensively
studied e.g. \cite{DILEO}, \cite{DilPast2}, \cite{PastSal}, \cite{Sal}.
Arbitrary almost Kenmotsu manifold can be locally deformed conformaly to
almost cosymplectic manifold. Almost Kenmotsu manifolds were generalized to
almost $\alpha $-Kenmotsu, $\alpha =const$, and subsequently to almost $%
\alpha $-cosymplectic manifolds.
\end{remark}


\section{Harmonic vector fields}

Let $(M,g)$ be smooth, oriented, connected pseudo-Riemannian manifold and $%
(TM,g^{S})$ its tangent bundle endowed with the Sasaki metric (also referred
to as Kaluza-Klein metric in Mathematical Physics) $g^{S}$. By definition,
the \textit{energy} of a smooth vector field $V$ on $M$ is the energy
corresponding $V:(M,g)\rightarrow (TM,g^{s}).$ When $M$ is compact, the 
\textit{energy} of $V$ is determined by%
\begin{equation*}
E(V)=\frac{1}{2}\int_{M}(tr_{g}V^{\ast }g^{s})dv=\frac{n}{2}vol(M,g)+\frac{1%
}{2}\int_{M}\left\Vert \nabla V\right\Vert ^{2}dv.
\end{equation*}%
The non-compact case, one can take into account over relatively compact
domains. It can be shown that $V:(M,g)\rightarrow (TM,g^{s})$ is harmonic
map if and only if%
\begin{equation}
tr\left[ R(\nabla .V,V).\right] =0,\ \nabla ^{\ast }\nabla V=0,
\label{NAMBLASTAR2}
\end{equation}%
where 
\begin{equation}
\nabla ^{\ast }\nabla V=-\sum_{i}\varepsilon _{i}(\nabla _{e_{i}}\nabla
_{e_{i}}V-\nabla _{\nabla _{e_{i}}e_{i}}V)  \label{NAMBLASTAR}
\end{equation}%
is the rough Laplacian with respect to a pseudo-orthonormal local frame $%
\left\{ e_{1},...,e_{n}\right\} $ on $(M,g)$ with $g(e_{i},e_{i})=%
\varepsilon _{i}=\pm 1$ for all indices $i=1,...,n.$

If ($M,g)$ is a compact Riemannian manifold, only parallel vector fields
define harmonic maps.

Next, for any real constant $\rho \neq 0$, let $\chi ^{\rho }(M)=\left\{
W\in \chi (M):\left\Vert W\right\Vert ^{2}=\rho \right\} .$We consider
vector fields $V\in $ $\chi ^{\rho }(M)$ which are critical points for the
energy functional $E\mid _{\chi ^{\rho }(M)}$, restricted to vector fields
of the same length. The Euler-Lagrange equations of this variational
condition yield that $V$ is a harmonic vector field if and only if%
\begin{equation}
\nabla ^{\ast }\nabla V\text{ is collinear to }V.  \label{NAMBLASTAR4}
\end{equation}

This characterization is well known in the Riemannian case ([3, 23, 25]).
Using same arguments in pseudo-Riemannian case, G. Calvaruso \cite%
{calvaruso1} proved that same result is still valid for vector fields of
constant length, if it is not lightlike.

Let $T_{1}M$ denote the unit tangent sphere bundle over $M$, and again by $%
g^{S}$ the metric induced on $T_{1}M$ by the Sasaki metric of $TM.$ Then, it
is shown that in \cite{ACP}, the map on $V:(M,g)\rightarrow (T_{1}M,g^{s})$
is harmonic if $V$ is a harmonic vector field and the additonial condition 
\begin{equation}
tr[R(\nabla .V,V).]=0  \label{TM}
\end{equation}%
is satisfied. G. Calvaruso \cite{calvaruso1} also investigated harmonicity
properties for left-invariant vector fields on three-dimensional Lorentzian
Lie groups, obtaining several classification results and new examples of
critical points of energy functionals.

In the non-compact case, conditions (\ref{NAMBLASTAR2}) and (\ref%
{NAMBLASTAR4}) are respectively taken as definitions of harmonic vector
fields and of vector fields defining harmonic maps.

Recently, D. Perrone proved that the characteristic vector field of an
almost cosymplectic three-manifold is minimal if and only if it is an
eigenvector of the Ricci operator.

\begin{theorem}
\label{HARMONIC0}Let ($M^{2n+1},\phi ,\xi ,\eta ,g)$ be an almost $\alpha $%
-para-Kenmotsu manifold. Then 
\begin{equation*}
\bar{\Delta}\xi =-\nabla ^{\ast }\nabla \xi =(2n\alpha ^{2}-tr(h^{2}))\xi
-Q\xi _{\mid \ker \eta }.
\end{equation*}
\end{theorem}

\begin{proof}
Now, let $(e_{i},\phi e_{i},\xi ),i=1,...,n,$ be a local orthogonal $\phi $%
-basis. Then we obtain%
\begin{eqnarray*}
\bar{\Delta}\xi &=&-\dsum\limits_{i=1}^{n}(\nabla _{e_{i}}\nabla _{e_{i}}\xi
-\nabla _{\nabla _{e_{i}}e_{i}}\xi -\nabla _{\phi e_{i}}\nabla _{\phi
e_{i}}\xi +\nabla _{\nabla _{\phi e_{i}}\phi e_{i}}\xi ) \\
&=&-\dsum\limits_{i=1}^{n}((\nabla _{e_{i}}\nabla \xi )e_{i}-(\nabla _{\phi
e_{i}}\nabla \xi )\phi e_{i}) \\
&&\overset{(\ref{7})}{=}\dsum\limits_{i=1}^{n}((\nabla
_{e_{i}}A)e_{i}-(\nabla _{\phi e_{i}}A)\phi e_{i}) \\
&=&-div\phi h+2n\alpha ^{2}\xi
\end{eqnarray*}%
By (\ref{SZETA}) and (\ref{SZETAZETA}) we get%
\begin{equation*}
\bar{\Delta}\xi =(2n\alpha ^{2}-tr(h^{2}))\xi -Q\xi _{\mid \ker \eta }.
\end{equation*}%
This ends the proof.
\end{proof}

As an immediate consequence of Theorem \ref{HARMONIC0} we obtain following
theorem.

\begin{theorem}
\label{HARMONIC}Let ($M^{2n+1},\phi ,\xi ,\eta ,g)$ be an almost $\alpha $%
-para-Kenmotsu manifold. $\xi $ is a harmonic vector field if and only if
the characteristic vector field is an eigenvector of the Ricci operator.
\end{theorem}

\section{Conformal and $D$-homothetic deformations.}

Let $M^{2n+1}$ be an almost $\alpha $-paracosymplectic manifold and $(\phi
,\xi ,\eta ,g)$ be almost $\alpha $-paracosymplectic structure. Let $%
\mathcal{R}_{\eta }(M^{2n+1})$ be the set of the locally defined smooth
functions $f$ on $M^{2n+1}$ such that $df\wedge \eta =0$, whenever $df$ is
defined.

Let $M^{2n+1}$ be an almost paracontact metric manifold. Let $f$ be a
function on $M^{2n+1}$, $f>0$ everywhere. Consider a deformation of the
structure 
\begin{equation}
\phi \mapsto \phi ^{\prime }=\phi ,\quad \xi \mapsto \xi ^{\prime }=\frac{1}{%
f}\xi ,\quad \eta \mapsto \eta ^{\prime }=f\eta ,\quad g\mapsto g^{\prime
}=fg,
\end{equation}%
we call $(\phi ^{\prime },\xi ^{\prime },\eta ^{\prime },g^{\prime })$ for
rather obvious reasons conformal deformation of $(\phi ,\xi ,\eta ,g)$.
Respectively we say that almost paracontact metric manifold $(M^{2n+1},\phi
^{\prime },\xi ^{\prime },\eta ^{\prime })$ is conformal to $(M^{2n+1},\phi
,\xi ,\eta ,g)$. Almost paracontact metric manifolds $(M^{2n+1},\phi ,\xi
,\eta ,g)$ and $(M^{2n+1},\phi ^{\prime },\xi ^{\prime },\eta ^{\prime
},g^{\prime })$ are called locally conformal if there is an open covering $%
(U_{i})_{i\in I}$, $M^{2n+1}=\bigcup U_{i}$, such that almost paracontact
metric manifolds $(U_{i},\phi |_{U_{i}},\xi |_{U_{i}},\eta
|_{U_{i}},g|_{U_{i}})$ and $(U_{i},\phi ^{\prime }|_{U_{i}},\xi ^{\prime
}|_{U_{i}},\eta ^{\prime }|_{U_{i}},g^{\prime }|_{U_{i}})$ are conformal.

\begin{theorem}
Arbitrary almost $\alpha $-paracosymplectic manifold $(M^{2n+1},\phi ,\xi
,\eta ,g)$, $n\geqslant 2$ is locally conformal to an almost
paracosymplectic manifold. In the other words near each point $p\in M^{2n+1}$
there is defined function $u$, such that a structure $(\phi ^{\prime },\xi
^{\prime },\eta ^{\prime },g^{\prime })$ 
\begin{equation}
\phi ^{\prime }=\phi ,\quad \xi ^{\prime }=e^{2u}\xi ,\quad \eta ^{\prime
}=e^{-2u}\eta ,\quad g^{\prime }=e^{-2u}g,  \label{confd}
\end{equation}%
is almost paracosymplectic. The function $u$ is unique up to additive
constant and $\alpha \eta =du$.
\end{theorem}

\begin{proof}
Let $u$ be a local function defined near a given point $p\in M^{2n+1}$. Then
we directly verify that the fundamental form of the structure $(\varphi
^{\prime },\xi ^{\prime },\eta ^{\prime },g^{\prime })$ is closed if and
only if $du=\alpha \eta $. Indeed $\Phi ^{\prime }(X,Y)=g^{\prime }(\phi
^{\prime }X,Y)=e^{-2u}g(\phi X,Y)=$ $e^{-2u}\Phi (X,Y$ and 
\begin{equation}
d\Phi ^{\prime }=-2e^{-2u}du\wedge \Phi +e^{-2u}d\Phi =2e^{-2u}(-du+\alpha
\eta )\wedge \Phi ,
\end{equation}%
by the Lemma \textbf{\ref{podzero}} $d\Phi ^{\prime }$ vanishes everywhere
only if the 1-form $-du+\alpha \eta $ is identically zero. Notice that in
the case when dimension of $M^{2n+1}$ is $\geqslant 5$, such function always
locally exists for according to the Proposition \textbf{\ref{properties1}}$%
(vii)$ the 1-form $\beta =\alpha \eta $ is closed everywhere $d\beta
=d\alpha \wedge \eta =f\eta \wedge \eta =0$, so from the Poincare lemma we
have a local solution. Similarly we obtain $d\eta ^{\prime
}=-2e^{-2u}du\wedge \eta =-2e^{-2u}\alpha \eta \wedge \eta =0$. Thus let $U$
be an open set, $p\in U$, such that the function $u$ is defined on $U$, then
the manifold $(U,\phi ^{\prime },\xi ^{\prime },\eta ^{\prime },g^{\prime })$
is almost para-cosymplectic.
\end{proof}

Consider a $D_{\gamma ,\beta }$-homothetic deformation of $(\phi ,\xi ,\eta
,g)$ into an almost paracontact metric structure $(\tilde{\phi},\tilde{\xi},%
\tilde{\eta},\tilde{g})$ defined as%
\begin{equation}
\tilde{\phi}=\phi ,\text{ \ \ }\tilde{\xi}=\frac{1}{\beta }\xi ,\text{ }%
\tilde{\eta}=\beta \eta ,\text{ \ }\tilde{g}=\gamma g+(\beta ^{2}-\gamma
)\eta \otimes \eta ,  \label{deformation}
\end{equation}%
where $\gamma $ is positive constant and $\beta \in \mathcal{R}_{\eta
}(M^{2n+1}),$ $\beta \neq 0$ at any point of $M^{2n+1}.$ Since $d\beta
\wedge \eta =0,$ it follows that%
\begin{equation*}
d\tilde{\eta}=d\beta \wedge \eta +\beta d\eta =0,
\end{equation*}%
and moreover $d\tilde{\Phi}=2(\frac{\gamma }{\beta })\tilde{\eta}\wedge 
\tilde{\Phi},$ since fundamental two forms $\Phi $, $\tilde{\Phi}$ are
related by $\tilde{\Phi}=\gamma \Phi .$ So, deformed structure $(\tilde{\phi}%
,\tilde{\xi},\tilde{\eta},\tilde{g})$ can be writen as%
\begin{equation*}
\tilde{\Phi}=\gamma \Phi ,\text{ \ }d\tilde{\eta}=0,\text{ }d\tilde{\Phi}=2%
\frac{\gamma }{\beta }\tilde{\eta}\wedge \tilde{\Phi}\text{\ ,}
\end{equation*}%
for $d\beta =d\beta (\xi )\eta $ and $\frac{\gamma }{\beta }\in \mathcal{R}%
_{\eta }(M^{2n+1}).$

Thus a $D_{\gamma ,\beta }$-homothetic deformation of an almost $\alpha $%
-paracosymplectic structure $(\phi ,\xi ,\eta ,g)$ gives a new almost $(%
\frac{\gamma }{\beta })$-paracosymplectic structure $(\tilde{\phi},\tilde{\xi%
},\tilde{\eta},\tilde{g})$ on the same manifold.

Following the definition of locally conformal almost paracontact metric
manifolds we define the notion of locally $D_{\gamma ,\beta }$-homothetic
almost $\alpha $-paracosymplectic manifolds. By the Proposition \textbf{\ref%
{properties1}}$(vii)$ we have the following

\begin{theorem}
\label{Dhomev} An almost $\alpha $-paracosymplectic manifold $(M^{2n+1},\phi
,\xi ,\eta ,g)$, $n\geqslant 2$ is locally $D_{\gamma ,\alpha }$-homothetic
to an almost para-Kenmotsu manifold on the set $U:\alpha \neq 0$.
\end{theorem}

\begin{proposition}
\label{nambla}Let $(\tilde{\phi},\tilde{\xi},\tilde{\eta},\tilde{g})$ be an
almost $\alpha $-paracosymplectic structure obtained from $(\phi ,\xi ,\eta
,g)$ by a $D_{\gamma ,\beta }$-homothetic deformation. Then we have the
following relationship between the Levi-Civita connections $\tilde{\nabla}$%
and $\nabla .$%
\begin{equation}
\tilde{\nabla}_{X}Y=\nabla _{X}Y-\left( \frac{\beta ^{2}-\gamma }{\beta ^{2}}%
\right) g(\mathcal{A}X,Y)\xi +\frac{d\beta (\xi )}{\beta }\eta (Y)\eta
(X)\xi .  \label{i00}
\end{equation}
\end{proposition}

\begin{proof}
By Kozsul's formula we have%
\begin{eqnarray*}
2\tilde{g}(\tilde{\nabla}_{X}Y,Z) &=&X\tilde{g}(Y,Z)+Y\tilde{g}(X,Z)-Z\tilde{%
g}(X,Y) \\
&&+\tilde{g}(\left[ X,Y\right] ,Z)+\tilde{g}(\left[ Z,X\right] ,Y)+\tilde{g}(%
\left[ Z,Y\right] ,X),
\end{eqnarray*}%
for any vector fields $X,Y,Z.$ By using $\tilde{g}=\gamma g+(\beta
^{2}-\gamma )\eta \otimes \eta $ in the last equation we obtain%
\begin{eqnarray}
2\tilde{g}(\tilde{\nabla}_{X}Y,Z) &=&2\gamma g(\nabla _{X}Y,Z)+2\beta d\beta
(\xi )\eta (X)\eta (Y)\eta (Z)  \label{i0} \\
&&+2(\beta ^{2}-\gamma )\left[ \eta (\nabla _{X}Y)\eta (Z)+g(Y,\nabla
_{X}\xi )\eta (Z)\right] .  \notag
\end{eqnarray}%
Moreover we have%
\begin{equation}
2\tilde{g}(\tilde{\nabla}_{X}Y,Z)=2\gamma g(\tilde{\nabla}_{X}Y,Z)+2(\beta
^{2}-\gamma )\eta (\tilde{\nabla}_{X}Y)\eta (Z)  \label{i1}
\end{equation}%
and%
\begin{equation}
\eta (\tilde{\nabla}_{X}Y)=\frac{1}{\beta ^{2}}\tilde{g}(\tilde{\nabla}%
_{X}Y,\xi ).  \label{i2}
\end{equation}%
Using (\ref{i1}) and (\ref{i2}) in (\ref{i0}) we find%
\begin{eqnarray}
\gamma g(\tilde{\nabla}_{X}Y,Z)+\frac{(\beta ^{2}-\gamma )}{\beta ^{2}}%
\tilde{g}(\tilde{\nabla}_{X}Y,\xi )\eta (Z) &=&\gamma g(\nabla
_{X}Y,Z)+\beta d\beta (\xi )\eta (X)\eta (Y)\eta (Z)  \notag \\
&&+(\beta ^{2}-\gamma )\left[ \eta (\nabla _{X}Y)\eta (Z)+g(Y,\nabla _{X}\xi
)\eta (Z)\right] .  \label{i3}
\end{eqnarray}%
Setting $Z=\xi $ in (\ref{i0}) we get%
\begin{eqnarray}
\tilde{g}(\tilde{\nabla}_{X}Y,\xi ) &=&\gamma g(\nabla _{X}Y,\xi )+\beta
d\beta (\xi )\eta (X)\eta (Y)  \label{i4} \\
&&+(\beta ^{2}-\gamma )\left[ \eta (\nabla _{X}Y)+g(Y,\nabla _{X}\xi )\right]
.  \notag
\end{eqnarray}%
Using (\ref{i4}) in (\ref{i3}), by a direct computation we have (\ref{i00}).
\end{proof}

\begin{proposition}
\label{relations}Let $(\tilde{\phi},\tilde{\xi},\tilde{\eta},\tilde{g})$ be
an almost $\alpha $-paracosymplectic structure obtained from $(\phi ,\xi
,\eta ,g)$ by a $D_{\gamma ,\beta }$-homothetic deformation. Then the
following identities hold:%
\begin{equation}
\mathcal{A}^{^{\prime }}X=\frac{1}{\beta }\mathcal{A}X,  \label{i5}
\end{equation}%
\begin{equation}
\tilde{h}X=\frac{1}{\beta }hX,  \label{i6}
\end{equation}%
\begin{equation}
\tilde{R}(X,Y)\tilde{\xi}=\frac{1}{\beta }R(X,Y)\xi +\frac{1}{\beta ^{2}}%
d\beta (\xi )\left[ \eta (X)\mathcal{A}Y-\eta (Y)\mathcal{A}X\right] ,
\label{i7}
\end{equation}%
for any vector fields $X,Y,Z.$
\end{proposition}

By using (\ref{7}), (\ref{deformation}) and (\ref{i00}) we obtain%
\begin{equation*}
\mathcal{A}^{^{\prime }}X=\frac{X(\beta )}{\beta ^{2}}\xi -\frac{1}{\beta }%
\nabla _{X}\xi -\frac{1}{\beta ^{2}}d\beta (\xi )\eta (X)\xi .
\end{equation*}%
By virtue of the definition $\beta $, the last equation reduces to (\ref{i5}%
). (\ref{i6}) follows from (\ref{deformation}) by using the properties of $%
h. $ First, from (\ref{deformation}) and (\ref{i00}) we have%
\begin{eqnarray}
\tilde{\nabla}_{X}\tilde{\nabla}_{Y}\tilde{\xi} &=&\nabla _{X}\tilde{\nabla}%
_{Y}\tilde{\xi}-\frac{(\beta ^{2}-\gamma )}{\beta ^{2}}g(\mathcal{A}X,\tilde{%
\nabla}_{Y}\tilde{\xi})\xi +\frac{1}{\beta }d\beta (\xi )\eta (X)\eta (%
\tilde{\nabla}_{Y}\tilde{\xi})\xi ,  \label{i8} \\
\tilde{\nabla}_{Y}\tilde{\xi} &=&\frac{1}{\beta }\nabla _{Y}\xi \text{\ .}
\label{i9}
\end{eqnarray}%
Using the properties of $\mathcal{A}$ and (\ref{i9}), (\ref{i8}) reduces to%
\begin{equation}
\tilde{\nabla}_{X}\tilde{\nabla}_{Y}\tilde{\xi}=\frac{X(\beta )}{\beta ^{2}}%
\mathcal{A}Y+\frac{1}{\beta }\nabla _{X}\nabla _{Y}\xi +\frac{(\beta
^{2}-\gamma )}{\beta ^{3}}g(\mathcal{A}X,\mathcal{A}Y)\xi .  \label{i10}
\end{equation}%
Then by (\ref{i9}) and (\ref{i10}) we obtain by a straightforward calculation%
\begin{eqnarray*}
\tilde{R}(X,Y)\tilde{\xi} &=&\tilde{\nabla}_{X}\tilde{\nabla}_{Y}\xi -\tilde{%
\nabla}_{Y}\tilde{\nabla}_{X}\xi -\tilde{\nabla}_{\left[ X,Y\right] }\tilde{%
\xi} \\
&=&\frac{X(\beta )}{\beta ^{2}}\mathcal{A}Y-\frac{Y(\beta )}{\beta ^{2}}%
\mathcal{A}X+\frac{1}{\beta }R(X,Y)\xi
\end{eqnarray*}%
which gives (\ref{i7}).

\section{\protect\bigskip Almost $\protect\alpha $-para-Kenmotsu $(\protect%
\kappa ,\protect\mu ,\protect\nu )$-Spaces}

In this section we study almost $\alpha $-para-Kenmotsu manifolds under
assumption that the curvature $R(X,Y)Z$ satisfies so called $(\kappa ,\mu
,\nu )$-nullity condition, i.e.%
\begin{equation}
R(X,Y)\xi =\eta (Y)BX-\eta (X)BY,  \label{BI}
\end{equation}%
where $B$ is a $(1,1)$-tensor field defined by%
\begin{equation}
BX=\kappa \phi ^{2}X+\mu hX+\nu \phi hX  \label{B2}
\end{equation}%
for $\kappa ,\mu ,\nu \in \mathcal{R}_{\eta }(M^{2n+1})$. Particularly $B\xi
=0.$

If an almost $\alpha $-paracosymplectic manifold satisfies (\ref{BI}), then
the manifold is said to be almost $\alpha $-paracosymplectic $(\kappa ,\mu
,\nu )$-space and $(\phi ,\xi ,\eta ,g)$ be called almost $\alpha $%
-paracosymplectic $(\kappa ,\mu ,\nu )$-structure.

Using (\ref{i7}) and after some calculations one can prove following
proposition.

\begin{proposition}
\label{relations2}Under the same assumptions of Proposition \ref{relations},
if $(\phi ,\xi ,\eta ,g)$ is an almost $\alpha $-para-Kenmotsu $(\kappa ,\mu
,\upsilon )$-structure then $(\tilde{\phi},\tilde{\xi},\tilde{\eta},\tilde{g}%
)$ is an almost $(\frac{\gamma }{\beta })$-paracosymplectic $(\tilde{\kappa},%
\tilde{\mu},\tilde{\nu})$ structure, where 
\begin{equation*}
\tilde{\kappa}=\frac{\kappa }{\beta ^{2}}+\frac{\alpha }{\beta ^{3}}d\beta
(\xi ),\text{ \ \ }\tilde{\mu}=\frac{\mu }{\beta },\text{ \ \ }\tilde{\nu}=%
\frac{\nu }{\beta }+\frac{d\beta (\xi )}{\beta ^{2}};\text{ \ }\tilde{\kappa}%
,\tilde{\mu},\tilde{\nu}\in \mathcal{R}_{\tilde{\eta}}(M^{2n+1}).
\end{equation*}
\end{proposition}

For an almost $\alpha $-paracosymplectic $(\kappa ,\mu ,\nu )$-space we may
consider a scalar invariant with respect to the $D_{\gamma ,\beta }$%
-homothety, that is a function $I(\alpha ,\kappa ,\mu ,\nu )$ with the
property that $I(\alpha ,\kappa ,\mu ,\nu )=I(\alpha ^{\prime },\kappa
^{\prime },\mu ^{\prime },\nu ^{\prime })$ for arbitrary $D_{\gamma ,\beta }$%
-homothety. In the case $\mu \neq 0$ by direct computations we find that 
\begin{equation}
I_{0}(\alpha ,\kappa ,\mu ,\nu )=\dfrac{\kappa -\alpha \nu }{\mu ^{2}},
\end{equation}%
is an invariant. An almost $\alpha $-paracosymplectic space will be called
of constant $I_{0}$-type if $\mu \neq 0$ and $I_{0}=const$ is a constant.

\begin{proposition}
\label{irem}Let $(M^{2n+1},\phi ,\xi ,\eta ,g)$ be an almost $\alpha $%
-para-Kenmotsu $(\kappa ,\mu ,\nu )$-space. Then the following identities
hold:%
\begin{equation}
l=\kappa \phi ^{2}+\mu h+\nu \phi h,  \label{ib1}
\end{equation}%
\begin{equation}
l\phi -\phi l=2\mu h\phi -2\nu h,  \label{ib2}
\end{equation}%
\begin{equation}
h^{2}=(\kappa +\alpha ^{2})\phi ^{2},  \label{ib3}
\end{equation}%
\begin{equation}
\nabla _{\xi }h=-(2\alpha +\nu )h+\mu h\phi ,  \label{ib4}
\end{equation}%
\begin{equation}
\nabla _{\xi }h^{2}=-2(2\alpha +\nu )(\kappa +\alpha ^{2})\phi ^{2},
\label{ib5}
\end{equation}%
\begin{equation}
\xi (\kappa )=-2(2\alpha +\nu )(\kappa +\alpha ^{2}),  \label{ib6}
\end{equation}%
\begin{eqnarray}
R(\xi ,X)Y &=&\kappa (g(X,Y)\xi -\eta (Y)X)+\mu (g(X,hY)\xi -\eta (Y)hX)
\label{ib7} \\
&&+\nu (g(X,\phi hY)\xi -\eta (Y)\phi hX),  \notag
\end{eqnarray}%
\begin{equation}
Q\xi =2n\kappa \xi ,  \label{ib8}
\end{equation}%
\begin{equation}
(\nabla _{X}\phi )Y=g(Y,hX+\alpha \phi X)\xi -\eta (Y)(hX+\alpha \phi X),
\label{ib9}
\end{equation}%
\begin{eqnarray}
(\nabla _{X}\phi h)Y-(\nabla _{Y}\phi h)X &=&(\kappa +\alpha ^{2})(\eta
(Y)X-\eta (X)Y)+\mu (\eta (Y)hX-\eta (X)hY)  \notag \\
&&+(\nu +\alpha )(\eta (Y)\phi hX-\eta (X)\phi hY),  \label{ib10}
\end{eqnarray}%
\begin{eqnarray}
(\nabla _{X}h)Y-(\nabla _{Y}h)X &=&(\kappa +\alpha ^{2})(\eta (Y)\phi X-\eta
(X)\phi Y+2g(Y,\phi X)\xi )+\mu (\eta (Y)\phi hX-\eta (X)\phi hY)  \notag \\
&&+(\nu +\alpha )(\eta (Y)hX-\eta (X)hY),  \label{ib11}
\end{eqnarray}%
for all vector fields $X,Y$ on $M^{2n+1}.$
\end{proposition}

\begin{proof}
From (\ref{BI}) we get%
\begin{equation}
lX=R(X,\xi )\xi =\kappa (X-\eta (X)\xi )+\mu hX+\nu \phi hX  \label{ib12}
\end{equation}%
which gives (\ref{ib1}). By replacing $X$ by $\phi X$ in (\ref{ib12}), we
have%
\begin{equation*}
l\phi X=\kappa \phi X+\mu h\phi X-\nu hX.
\end{equation*}%
By applying now $\phi $ to (\ref{ib12}), we obtain%
\begin{equation*}
\phi lX=\kappa \phi X-\mu h\phi X+\nu hX.
\end{equation*}%
(\ref{ib2}) comes from the last two equations. From (\ref{ib12}) we easily
get%
\begin{equation}
\phi l\phi X=\kappa \phi ^{2}X-\mu hX-\nu \phi hX.  \label{ib13}
\end{equation}%
Then by (\ref{ib12}) and (\ref{ib13}) we obtain%
\begin{equation*}
-lX-\phi l\phi X=2(\alpha ^{2}\phi ^{2}X-h^{2}X).
\end{equation*}%
Comparing this equation with (\ref{iremmmm}), we have (\ref{ib3}). (\ref{ib4}%
) can be easily get from using (\ref{ib3}) in (\ref{iremmm}). From (\ref{ib3}%
) we find%
\begin{equation*}
\nabla _{\xi }h^{2}=(\nabla _{\xi }h)h+h(\nabla _{\xi }h)=-2(2\alpha +\nu
)h^{2}.
\end{equation*}%
(\ref{ib5}) comes from using (\ref{ib3}) in the last equation. One can
easily get (\ref{ib6}) by differentiating (\ref{ib3}) along $\xi $. Using $%
g(R(\xi ,X)Y,Z)=g(R(Y,Z)\xi ,X)$ and (\ref{BI}) we have%
\begin{eqnarray}
g(R(Y,Z)\xi ,X) &=&\kappa (\eta (Z)g(X,Y)-\eta (Y)g(X,Z))+\mu (\eta
(Z)g(X,hY)-\eta (Y)g(X,hZ))  \notag \\
&&+\nu (\eta (Z)g(X,\phi hY)-\eta (Y)g(X,\phi hZ)).  \label{ib14}
\end{eqnarray}%
The last equation completes the proof of (\ref{ib7}). Then the definition of
the Ricci operator directly gives (\ref{ib8}). For (\ref{ib9}), by virtue of
(\ref{ib14}), the left hand side of Eq. (\ref{irem5}) can be written as%
\begin{equation*}
2\kappa (\eta (Z)g(X,Y)-\eta (Y)g(X,Z)).
\end{equation*}%
So, (\ref{irem5}) reduces to%
\begin{eqnarray*}
2\kappa (\eta (Z)g(X,Y)-\eta (Y)g(X,Z)) &=&2(\nabla _{hX}\Phi )(Y,Z)+2\alpha
^{2}\eta (Y)g(X,Z) \\
&&-2\alpha ^{2}\eta (Z)g(X,Y)-2\alpha \eta (Z)g(\phi hX,Y)+2\alpha \eta
(Y)g(\phi hX,Z).
\end{eqnarray*}%
From the last equation we have%
\begin{equation*}
(\nabla _{hX}\Phi )(Y,Z)=(\kappa +\alpha ^{2})(\eta (Z)g(X,Y)-\eta
(Y)g(X,Z))+\alpha (\eta (Z)g(\phi hX,Y)-\eta (Y)g(\phi hX,Z)).
\end{equation*}%
By replacing $X$ by $hX$ in that equation, using (\ref{ib3}) and the
relation (\ref{i}), we get%
\begin{equation}
g((\nabla _{X}\phi )Y,Z)=(\eta (Z)g(hX,Y)-\eta (Y)g(hX,Z))+\alpha (\eta
(Z)g(\phi X,Y)-\eta (Y)g(\phi X,Z)).  \label{ib15}
\end{equation}%
Then (\ref{ib9}) follows from (\ref{ib15}). On the other hand Eq. (\ref{ib9}%
) can be written as 
\begin{equation*}
(\nabla _{X}\phi )Y=-g(\phi \mathcal{A}X,Y)\xi +\eta (Y)\phi \mathcal{A}X.
\end{equation*}%
From (\ref{b3}) we find%
\begin{equation*}
(\nabla _{X}\phi h)Y-(\nabla _{Y}\phi h)X=R(X,Y)\xi -\alpha ^{2}(\eta
(X)Y-\eta (Y)X)-\alpha (\eta (X)\phi hY-\eta (Y)\phi hX).
\end{equation*}%
Using (\ref{BI}) in the last equation we obtain (\ref{ib10}). One can easily
show that 
\begin{subequations}
\begin{equation}
(\nabla _{X}\phi h)Y-(\nabla _{Y}\phi h)X=(\nabla _{X}\phi )hY-(\nabla
_{Y}\phi )hX+\phi ((\nabla _{X}h)Y-(\nabla _{Y}h)X).  \label{cm1}
\end{equation}%
By replacing $Y$ by $hY$ in (\ref{ib9}), we get 
\begin{equation}
(\nabla _{X}\phi )hY=g(hY,hX+\alpha \phi X)\xi .  \label{m3}
\end{equation}%
From (\ref{b3}) and (\ref{BI}), we have 
\end{subequations}
\begin{equation}
\begin{array}{c}
\kappa (\eta (Y)X-\eta (X)Y)+\mu (\eta (Y)hX-\eta (X)hY)+\nu (\eta (Y)\phi
hX-\eta (X)\phi hY)= \\[4pt] 
\alpha ^{2}(\eta (X)Y-\eta (Y)X)+\alpha (\eta (X)\phi hY-\eta (Y)\phi
hX)+(\nabla _{X}\phi h)Y-(\nabla _{Y}\phi h)X.%
\end{array}
\label{ib16}
\end{equation}%
After using (\ref{cm1}) and (\ref{m3}) in (\ref{ib16}), we obtain%
\begin{equation}
\begin{array}{l}
\kappa (\eta (Y)X-\eta (X)Y)+\mu (\eta (Y)hX-\eta (X)hY)+\nu (\eta (Y)\phi
hX-\eta (X)\phi hY)= \\[4pt] 
\alpha ^{2}(\eta (X)Y-\eta (Y)X)+\alpha (\eta (X)\phi hY-\eta (Y)\phi
hX)+\alpha g(\phi X,hY)\xi \\[4pt] 
-\alpha g(\phi Y,hX)\xi +\phi ((\nabla _{X}h)Y-(\nabla _{Y}h)X).%
\end{array}
\label{ib17}
\end{equation}%
Then by applying $\phi $ to (\ref{ib17}) we have (\ref{ib11}).
\end{proof}

In the following result we additionally assume that the set $U:\alpha \neq 0$
is dense.

\begin{theorem}
\label{REM} An almost $\alpha $-para-Kenmotsu $(\kappa ,\mu ,\nu )$-spaces $%
(M^{2n+1},\phi ,\xi ,\eta ,g)$, satisfy the para-Kaehlerian structure
condition.
\end{theorem}

\begin{proof}
We only need to prove this if $n\geqslant 2$ for arbitrary $3$-dimensional $%
\alpha $-para-Kenmotsu manifold satisfies has para-Kaehler leaves. According
to the Theorem \textbf{\ref{Dhomev}} the $\alpha $-paracosymplectic manifold 
$(U,\phi |_{U},\xi |_{U},\eta |_{U},g|_{U})$ is $D_{1,\alpha }$-homotetic to
almost para-Kenmotsu manifold $(U,\phi ^{\prime },\xi ^{\prime },\eta
^{\prime },g^{\prime })$, (cf. \ref{deformation}). If necessary we restrict
our attention to connected components of $U$. Now Eq.(\ref{ib9}) tells us
that $U$ viewed as almost para-Kenmotsu manifold has para-Kaehler leaves.
Moreover we notice that arbitrary $D_{\gamma ,\beta }$-homothety preserves
this property thus we conclude that the original structure $(\phi ,\xi ,\eta
,g)$ also satisfies the para-Kaehlerian structure condition. Finally if $%
(\nabla _{X}\phi )Y$ satisfies (\ref{ib9}) on $U$, then this identity must
be satisfied everywhere on $M^{2n+1}$.
\end{proof}

For manifolds with constant sectional curvature $c$, $R(X,Y)\xi =c(\eta
(Y)X-\eta (X))$ thus in our terminology these manifolds are almost
paracosymplectic $(c,0,0)$-spaces.

\begin{corollary}
An almost $\alpha $-para-Kenmotsu manifold of constant sectional curvature
has para-Kaehlerian leaves.
\end{corollary}

\begin{corollary}
\label{QQ}Let $(M^{2n+1},\phi ,\xi ,\eta ,g)$ be an almost $\alpha $%
-para-Kenmotsu $(\kappa ,\mu ,\nu )$-spaces. Then%
\begin{equation}
Q\phi -\phi Q=2\mu h\phi X-2(\nu +2\alpha (1-n))hX.  \label{Q12}
\end{equation}
\end{corollary}

\begin{proof}
Using (\ref{BI}) and $\phi h=-h\phi $ we obtain $l\phi -\phi l=2\mu h\phi
X-2\nu hX$. On the other hand from (\ref{ib8}) one can easily prove that
both $\eta \otimes \phi Q$ and $(\eta \circ Q\phi )\otimes \xi $ vanish. So (%
\ref{Q12}) follows from (\ref{Q1}).
\end{proof}

\begin{remark}
Manifolds which are conformal or locally conformal to cosymplectic manifolds
were studied by many authors e.g.\cite{CapDrag}, \cite{ChinMar}, \cite{Fal-1}%
, \cite{Falcitelli}, \cite{Olszak}, \cite{marrero}.
\end{remark}

\begin{remark}
$D_{\gamma,\beta}$-homoteties as they appear in almost contact metric
geometry are particular class of deformations considerd by S. Tanno \cite%
{Tanno}. The general deformation of a metric (Riemannian) has a form $%
g^{\prime }=\alpha g+ \omega\otimes\theta+\theta\otimes\omega+\beta\omega
\otimes\omega$ where $\omega$, $\beta$ are one-forms and $\alpha$, $\beta$
are functions, $\alpha>0$, $\alpha+\beta >0$. The paper seems to be nowadays
completely forgotten in the framework of almost contact metric geometry.
\end{remark}

\section{Classification of the $3$-dimensional almost $\protect\alpha $%
-para-Kenmotsu $(\protect\kappa ,\protect\mu ,\protect\nu )$-spaces}

In this section, different possibilities for the tensor field $h$ are
investigated. Thus we can comprehend the differences between the almost $%
\alpha $-para-Kenmotsu and almost $\alpha $-Kenmotsu cases by looking at the
possible Jordan forms of the tensor field $h.$

It is well known that a self-adjoint linear operator $\Psi $ of a Euclidean
space is always diagonalizable, but this is not the case for a self-adjoint
linear operator $\Psi $ for a Lorentzian inner product. It is known (\cite%
{a.z.petrov}, pp. 50-55) that self-adjoint linear operator of a vector space
with a Lorentzian inner product can be put into four possible canonical
forms. In particular, the matrix representation $g$ of the induced metric on 
$M_{1}^{3} $ is of Lorentz type, so the self-adjoint linear $\Psi $ of $%
M_{1}^{3}$ can be put into one of the following four forms with respect to
frames $\left\{ e_{1},e_{2},e_{3}\right\} $ at $T_{p}M_{1}^{3}$ where $%
T_{p}M_{1}^{3}$ is a tangent space to $M$ at $p$ \cite{MAGID},\cite{ONEILL}.

$\mathfrak{h}_{1}$-type$)$ $\ \ \ \ \ \ \ \ \ \ \ \ \ \ \ \ \Psi =\left( 
\begin{array}{ccc}
\lambda _{1} & 0 & 0 \\ 
0 & \lambda _{2} & 0 \\ 
0 & 0 & \lambda _{3}%
\end{array}%
\right) ,$ \ \ \ \ $\ \ \ \ \ \ g=\left( 
\begin{array}{ccc}
-1 & 0 & 0 \\ 
0 & 1 & 0 \\ 
0 & 0 & 1%
\end{array}%
\right) ,$

\bigskip

$\mathfrak{h}_{2}$-type$)$ \ \ \ \ \ \ \ \ \ \ \ \ \ \ \ \ $\Psi $ $=\left( 
\begin{array}{ccc}
\lambda & 0 & 0 \\ 
1 & \lambda & 0 \\ 
0 & 0 & \lambda _{3}%
\end{array}%
\right) ,$ \ \ \ \ \ \ \ \ \ \ \ \ \ \ $g=\left( 
\begin{array}{ccc}
0 & 1 & 0 \\ 
1 & 0 & 0 \\ 
0 & 0 & 1%
\end{array}%
\right) $,\ \ 

\bigskip

$\mathfrak{h}_{3}$-type$)$\ \ \ \ \ \ \ \ \ \ \ \ \ \ $\ \Psi =\left( 
\begin{array}{ccc}
\gamma & -\lambda & 0 \\ 
\lambda & \gamma & 0 \\ 
0 & 0 & \lambda _{3}%
\end{array}%
\right) ,$ \ \ \ \ $\ \ \ \ \ \ $\ $\ \ g=\left( 
\begin{array}{ccc}
-1 & 0 & 0 \\ 
0 & 1 & 0 \\ 
0 & 0 & 1%
\end{array}%
\right) ,\lambda \neq 0$,\ \ \ \ \ \ \ \ \ \ 

\bigskip

$\mathfrak{h}_{4}$-type$)$ $\ \ \ \ \ \ \ \ \ \ \ \ \ \ \ \ \Psi =\left( 
\begin{array}{ccc}
\lambda & 0 & 0 \\ 
0 & \lambda & 1 \\ 
1 & 0 & \lambda%
\end{array}%
\right) ,$ \ \ \ \ \ \ \ \ \ \ \ \ \ \ $g=\left( 
\begin{array}{ccc}
0 & 1 & 0 \\ 
1 & 0 & 0 \\ 
0 & 0 & 1%
\end{array}%
\right) $.

The matrices $g$ for types $\mathfrak{h}_{1})$ and $\mathfrak{h}_{3})$ are
with respect to an orthonormal basis of \ $T_{p}M_{1}^{3},$ whereas for
types $\mathfrak{h}_{2})$ and $\mathfrak{h}_{4})$ are with respect to a
pseudo-orthonormal basis. This is a basis $\left\{ e_{1},e_{2},e_{3}\right\} 
$ of \ $T_{p}M_{1}^{3}$ satisfying $%
g(e_{1},e_{1})=g(e_{2},e_{2})=g(e_{1},e_{3})=g(e_{2},e_{3})=0$ and $%
g(e_{1},e_{2})=g(e_{3},e_{3})=1.$\ 

Let $(M,\phi ,\xi ,\eta ,g)$ be a $3$-dimensional almost $\alpha $%
-paracosymplectic manifold .Then operator $h$ has \ following types.

$\mathfrak{h}_{1}$-type$)$ 
\begin{eqnarray*}
U_{1} &=&\left\{ p\in M\mid h(p)\neq 0\right\} \subset M \\
U_{2} &=&\left\{ p\in M\mid h(p)=0,\text{in a neighborhood of p}\right\}
\subset M
\end{eqnarray*}%
That $h$ is a smooth function on $M$ implies $U_{1}\cup U_{2}$ is an open
and dense subset of $M$, so any property satisfied in $U_{1}\cup U_{2}$ is
also satisfied in $M.$ For any point $p\in U_{1}\cup $ $U_{2}$ there exists
a local orthonormal $\phi $- basis $\{e,\phi e,\xi \}$ of smooth
eigenvectors of $h$ in a neighborhood of $p$, where $-g(e,e)=g(\phi e,\phi
e)=g(\xi ,\xi )=1.$ On $\ U_{1}$ we put $he=\lambda e,$ where $\lambda $ is
a non-vanishing smooth function. Since $trh=0$, we have $h\phi e=-\lambda
\phi e$. The eigenvalue function $\lambda $ is continuos on $\ M$ and smooth
on $U_{1}\cup U_{2}.$ So, $h$ \ has following form%
\begin{equation}
\left( 
\begin{array}{ccc}
\lambda & 0 & 0 \\ 
0 & -\lambda & 0 \\ 
0 & 0 & 0%
\end{array}%
\right)  \label{A1}
\end{equation}%
respect to local orthonormal $\phi $-basis $\{e,\phi e,\xi \}.$

$\mathfrak{h}_{2}$-type$)$ Using same methods in \cite{kupmur} one can
construct a local pseudo-orthonormal basis $\{e_{1},e_{2},e_{3}\}$ in a
neighborhood of $p$ where $%
g(e_{1},e_{1})=g(e_{2},e_{2})=g(e_{1},e_{3})=g(e_{2},e_{3})=0$ and $%
g(e_{1},e_{2})=g(e_{3},e_{3})=1$. Let $\mathcal{U}$ be the open subset of $M$
where $h\neq 0$. For every $p\in \mathcal{U}$ there exists an open
neighborhood of $p$ such that $he_{1}=e_{2},he_{2}=0,he_{3}=0$ and $\phi
e_{1}=\pm e_{1},$ $\phi e_{2}=\mp e_{2},\phi e_{3}=0$ and also $\xi =e_{3}$.
Thus the tensor $h$ has the form

\begin{equation}
\left( 
\begin{array}{ccc}
0 & 0 & 0 \\ 
1 & 0 & 0 \\ 
0 & 0 & 0%
\end{array}%
\right)  \label{A2}
\end{equation}%
relative a pseudo-orthonormal basis $\{e_{1},e_{2},e_{3}\}.$

$\mathfrak{h}_{3}$-type$)$ We can find a local orthonormal $\phi $-basis $%
\{e,\phi e,\xi \}$ in a neighborhood of $p$ where $-g(e,e)=g(\phi e,\phi
e)=g(\xi ,\xi )=1.$ Now, let $\mathcal{U}_{1}$ be the open subset of $M$
where $h\neq 0$ and let $\mathcal{U}_{2}$ be the open subset of points $p\in
M$ such that $h=0$ in a neighborhood of $p.$ $\mathcal{U}_{1}\cup $ $%
\mathcal{U}_{2}$ is an open subset of $M$. For every $p\in \mathcal{U}_{1}$
there exists an open neighborhood of $p$ such that $he=\lambda \phi e,h\phi
e=-\lambda e$ and $h\xi =0$ where $\lambda $ is a non-vanishing smooth
function. Since $trh=0,\ $the matrix form of $h$ is$\ $given$\ $by%
\begin{equation}
\left( 
\begin{array}{ccc}
0 & -\lambda & 0 \\ 
\lambda & 0 & 0 \\ 
0 & 0 & 0%
\end{array}%
\right)  \label{A3}
\end{equation}%
with respect to local orthonormal basis $\{e,\phi e,\xi \}.$

$\mathfrak{h}_{4}$-type$)$ Then a local pseudo-orthonormal basis $%
\{e_{1},e_{2},e_{3}\}$ is constructed in a neighborhood of $p$ where $%
g(e_{1},e_{1})=g(e_{2},e_{2})=g(e_{1},e_{3})=g(e_{2},e_{3})=0$ and $%
g(e_{1},e_{2})=g(e_{3},e_{3})=1.$ Since the tensor $h$ is $\mathfrak{h}_{4}$%
-type) (with respect to a pseudo-orthonormal basis $\{e_{1},e_{2},e_{3}\})$
then $he_{1}=\lambda e_{1}+e_{3},$ \ $he_{2}=\lambda e_{2}$ and $%
he_{3}=e_{2}+\lambda e_{3}.$ Since $%
0=trh=g(he_{1},e_{2})+g(he_{2},e_{1})+g(he_{3},e_{3})=3\lambda $, then $%
\lambda =0.$ We write $\xi =g(\xi ,e_{2})e_{1}+g(\xi ,e_{1})e_{2}+g(\xi
,e_{3})e_{3}$ respect to the pseudo-orthonormal basis $\{e_{1},e_{2},e_{3}\}$%
. Since $h\xi =0,$ we have $0=g(\xi ,e_{2})e_{3}+g(\xi ,e_{3})e_{2}.$ Hence
we get $\xi =g(\xi ,e_{1})e_{2}$ which leads to a contradiction with $g(\xi
,\xi )=1$. Thus, this case does not occur.

Since the proof of following lemma is similar to \cite{kupmur} we \ omit
proof of it.

\begin{lemma}
Let $(M,\phi ,\xi ,\eta ,g)$ be a $3$-dimensional almost $\alpha $%
-para-Kenmotsu manifold. \textit{Then a canonical form of }$h$ \textit{stays
constant in an open neighborhood of any point for }$h$\textit{. }
\end{lemma}

In a $3$-dimensional pseudo-Riemannian manifold case, the curvature tensor
can be written by%
\begin{equation}
R(X,Y)Z=g(Y,Z)QX-g(X,Z)QY+g(QY,Z)X-g(QX,Z)Y-\frac{r}{2}(g(Y,Z)X-g(X,Z)Y).
\label{THREE DIM CURVATURE}
\end{equation}%
for any $X,Y,Z\in \Gamma (TM).$

Using same procedure with \cite{PERRONE}, we have

\begin{lemma}
\label{K>-1}Let $(M,\phi ,\xi ,\eta ,g)$ be a $3$-dimensional almost $\alpha 
$-para-Kenmotsu manifold with $h$ of $\mathfrak{h}_{1}$\emph{\ }\textit{type}%
$.$Then for the covariant derivative on $\mathcal{U}_{1}$ the following
equations are valid 
\begin{eqnarray}
i)\text{ }\nabla _{e}e &=&\frac{1}{2\lambda }\left[ \sigma (e)-(\phi
e)(\lambda )\right] \phi e+\alpha \xi ,\text{ \ }ii)\text{ }\nabla _{e}\phi
e=\frac{1}{2\lambda }\left[ \sigma (e)-(\phi e)(\lambda )\right] e-\lambda
\xi ,\text{ }  \notag \\
\text{\ }iii)\text{ }\nabla _{e}\xi &=&\alpha e+\lambda \phi e,  \notag \\
iv)\text{ }\nabla _{\phi e}e &=&-\frac{1}{2\lambda }\left[ \sigma (\phi
e)+e(\lambda )\right] \phi e-\lambda \xi ,\text{ \ }v)\text{ }\nabla _{\phi
e}\phi e=-\frac{1}{2\lambda }\left[ \sigma (\phi e)+e(\lambda )\right] \phi
e-\alpha \xi ,\text{ \ }  \notag \\
vi)\nabla _{\phi e}\xi &=&\alpha \phi e-\lambda e  \notag \\
vii)\nabla _{\xi }e &=&a_{1}\phi e,\text{ \ }viii)\text{ }\nabla _{\xi }\phi
e=a_{1}e,\text{ \ \ }  \notag \\
ix)\text{ }[e,\xi ] &=&\alpha e+(\lambda -a_{1})\phi e,\text{ \ }x)\text{ }%
[\phi e,\xi ]=-(\lambda +a_{1})e+\alpha \phi e,\text{ \ }  \label{irem0} \\
xi)\text{ }[e,\phi e] &=&\frac{1}{2\lambda }\left[ \sigma (e)-(\phi
e)(\lambda )\right] e+\frac{1}{2\lambda }\left[ \sigma (\phi e)+e(\lambda )%
\right] \phi e,  \notag \\
xii)\nabla _{\xi }h &=&\xi (\lambda )s-2a_{1}h\phi ,\text{ \ \ }%
xiii)h^{2}-\alpha ^{2}\phi ^{2}=\frac{1}{2}S(\xi ,\xi )\phi ^{2}.  \notag
\end{eqnarray}%
where 
\begin{equation*}
\text{ }a_{1}=g(\text{ }\nabla _{\xi }e,\phi e),\text{ }\sigma =S(\xi
,.)_{\ker \eta }\text{\ .}
\end{equation*}
\end{lemma}

\begin{lemma}
\label{RICCI1}Let $(M,\phi ,\xi ,\eta ,g)$ be a $3$-dimensional almost $%
\alpha $-para-Kenmotsu manifold with $h$ of $\mathfrak{h}_{1}$\emph{\ }%
\textit{type}$.$Then the Ricci operator $Q$ is given by 
\begin{equation}
Q=(\frac{r}{2}+\alpha ^{2}-\lambda ^{2})I+(-\frac{r}{2}+3(\lambda
^{2}-\alpha ^{2}))\eta \otimes \xi -2\alpha \phi h-\phi (\nabla _{\xi
}h)+\sigma (\phi ^{2})\otimes \xi -\sigma (e)\eta \otimes e+\sigma (\phi
e)\eta \otimes \phi e.  \label{QH1}
\end{equation}
\end{lemma}

\begin{lemma}
\label{de2}Let $(M,\phi ,\xi ,\eta ,g)$ be a $3$-dimensional almost $\alpha $%
-para-Kenmotsu manifold with $h$ of $\mathfrak{h}_{2}$ type. Then for the
covariant derivative on $\mathcal{U}$ the following equations are valid 
\begin{eqnarray}
i)\text{ }\nabla _{e_{1}}e_{1} &=&-b_{1}e_{1}+\xi ,\text{ \ }ii)\text{ }%
\nabla _{e_{1}}e_{2}=b_{1}e_{2}-\alpha \xi ,\text{ \ }iii)\text{ }\nabla
_{e_{1}}\xi =\alpha e_{1}-e_{2},  \notag \\
iv)\text{ }\nabla _{e_{2}}e_{1} &=&-b_{2}e_{1}-\alpha \xi ,\text{ \ }v)\text{
}\nabla _{e_{2}}e_{2}=b_{2}e_{2},\text{ \ }vi)\text{ }\nabla _{e_{2}}\xi
=\alpha e_{2},  \notag \\
vii)\text{ }\nabla _{\xi }e_{1} &=&a_{2}e_{1},\text{ \ }viii)\text{ }\nabla
_{\xi }e_{2}=-a_{2}e_{2},\text{ \ \ }  \notag \\
ix)\text{ }[e_{1},\xi ] &=&(\alpha -a_{2})e_{1}-e_{2},\text{ \ }x)\text{ }%
[e_{2},\xi ]=(\alpha +a_{2})e_{2},\text{ \ }  \label{iremm0} \\
xi)\text{ }[e_{1},e_{2}] &=&b_{2}e_{1}+b_{1}e_{2}.  \notag \\
xii)\nabla _{\xi }h &=&-2a_{2}h\phi ,\text{ \ }xiii)\text{\ }h^{2}=0.  \notag
\end{eqnarray}%
where $a_{2}=g(\nabla _{\xi }e_{1},e_{2}),$ $b_{1}=g(\nabla
_{e_{1}}e_{2},e_{1})$ and $b_{2}=g(\nabla _{e_{2}}e_{2},e_{1})=-\frac{1}{2}%
\sigma (e_{1}).$
\end{lemma}

\begin{proof}
By $\nabla \xi =-\alpha ^{2}\phi +\phi h$, we obtain $iii),vi).$

Using pseudo-orthonormal basis $\{e_{1},e_{2},e_{3}=\xi \}$ with $\phi
e_{1}=e_{1},$ $\phi e_{2}=-e_{2}$, $\phi e_{3}=0$ we have%
\begin{eqnarray*}
\nabla _{e_{1}}e_{2} &=&g(\nabla _{e_{1}}e_{2},e_{2})e_{1}+g(\nabla
_{e_{1}}e_{2},e_{1})e_{2}+g(\nabla _{e_{1}}e_{2},\xi )\xi \\
&=&g(\nabla _{e_{1}}e_{2},e_{1})e_{2}-g(e_{2},\nabla _{e_{1}}\xi )\xi \\
&&\overset{iii)}{=}g(\nabla _{e_{1}}e_{2},e_{1})e_{2}-\alpha \xi \\
&=&b_{1}e_{2}-\alpha \xi .
\end{eqnarray*}

The proofs of other covariant derivative equalities are similar to $ii).$

Putting $X=e_{1},$ $Y=e_{2}$ and \ $Z=\xi $ in the equation (\ref{THREE DIM
CURVATURE}), we have%
\begin{equation}
R(e_{1},e_{2})\xi =-\sigma (e_{1})e_{2}+\sigma (e_{2})e_{1}.  \label{iremm2}
\end{equation}%
On the other hand, by using (\ref{b3}), we get 
\begin{eqnarray}
R(e_{1},e_{2})\xi &=&(\nabla _{e_{1}}\phi h)e_{2}-(\nabla _{e_{2}}\phi
h)e_{1}  \notag \\
&=&2b_{2}e_{2}.  \label{iremm3}
\end{eqnarray}%
Comparing (\ref{iremm3}) with (\ref{iremm2}), we obtain%
\begin{equation}
\sigma (e_{1})=-2b_{2},~\ \sigma (e_{2})=0=S(\xi ,e_{2}).\   \label{iremm3a}
\end{equation}%
Hence, the function $b_{2}$ is obtained from the last equation.
\end{proof}

\begin{lemma}
\label{RICCI2}Let $(M,\phi ,\xi ,\eta ,g)$ be a $3$-dimensional almost $%
\alpha $-para-Kenmotsu manifold with $h$ of $\mathfrak{h}_{2}$\emph{\ }%
\textit{type}$.$Then the Ricci operator $Q$ is given by 
\begin{equation}
Q=(\frac{r}{2}+\alpha ^{2})I-(\frac{r}{2}+3\alpha ^{2})\eta \otimes \xi
-2\alpha \phi h-\phi (\nabla _{\xi }h)+\sigma (\phi ^{2})\otimes \xi +\sigma
(e_{1})\eta \otimes e_{2}.  \label{QH2}
\end{equation}
\end{lemma}

\begin{proof}
From (\ref{THREE DIM CURVATURE}), we obtain%
\begin{equation*}
R(X,\xi )\xi =S(\xi ,\xi )X-S(X,\xi )\xi +QX-\eta (X)Q\xi -\frac{r}{2}%
(X-\eta (X)\xi ),
\end{equation*}%
for any vector field $X.$ By (\ref{iremm}) and (\ref{SZETAZETA}) the last
equation reduces to%
\begin{equation}
QX=\frac{1}{2}S(\xi ,\xi )\phi ^{2}X-2\alpha \phi hX-\phi (\nabla _{\xi
}h)X-S(\xi ,\xi )X+S(X,\xi )\xi +\eta (X)Q\xi +\frac{r}{2}(X-\eta (X)\xi ).
\label{QX1}
\end{equation}%
By setting $S(X,\xi )=S(\phi ^{2}X,\xi )+\eta (X)S(\xi ,\xi )$ in (\ref{QX1}%
), we have 
\begin{equation}
QX=\frac{S(\xi ,\xi )}{2}\phi ^{2}X-2\alpha \phi hX-\phi (\nabla _{\xi
}h)X-S(\xi ,\xi )X+S(\phi ^{2}X,\xi )\xi +\eta (X)S(\xi ,\xi )\xi +\eta
(X)Q\xi +\frac{r}{2}\phi ^{2}X.  \label{iremmm88}
\end{equation}%
On the other hand, the Ricci tensor $S$ can be written with respect to the
orthonormal basis $\{e_{1},e_{2},\xi \}$ as following 
\begin{equation}
Q\xi =\sigma (e_{1})e_{2}+S(\xi ,\xi )\xi  \label{iremmm99}
\end{equation}%
Using (\ref{iremmm99}) in (\ref{iremmm88}), we get%
\begin{eqnarray}
QX &=&\frac{1}{2}\left( r+2\alpha ^{2}\right) X-\frac{1}{2}(6\alpha
^{2}+r)\eta (X)\xi -2\alpha \phi h  \label{iremmm10} \\
&&-\phi (\nabla _{\xi }h)X+\sigma (\phi ^{2}X)\xi +\eta (X)\sigma
(e_{1})e_{2}+  \notag
\end{eqnarray}%
for arbitrary vector field $X$. This ends the proof.
\end{proof}

\begin{lemma}
\label{lem3}Let $(M,\phi ,\xi ,\eta ,g)$ be a $3$-dimensional almost $\alpha 
$-para-Kenmotsu manifold with $h$ of $\mathfrak{h}_{3}$\emph{\ }\textit{type}%
. Then for the covariant derivative on $\mathcal{U}_{1}$ the following
equations are valid 
\begin{eqnarray}
i)\text{ }\nabla _{e}e &=&b_{3}\phi e+(\alpha +\lambda )\xi ,\text{ \ }ii)%
\text{ }\nabla _{e}\phi e=b_{3}e,\text{ \ }iii)\text{ }\nabla _{e}\xi
=(\alpha +\lambda )e,  \notag \\
iv)\text{ }\nabla _{\phi e}e &=&b_{4}\phi e,\text{ \ }v)\text{ }\nabla
_{\phi e}\phi e=b_{4}e+(\lambda -\alpha )\xi ,\text{ \ }vi)\text{ }\nabla
_{\phi e}\xi =-(\lambda -\alpha )\phi e,  \notag \\
vii)\text{ }\nabla _{\xi }e &=&a_{3}\phi e,\text{ \ }viii)\text{ }\nabla
_{\xi }\phi e=a_{3}e,\text{ \ \ }  \notag \\
ix)\text{ }[e,\xi ] &=&(\alpha +\lambda )e-a_{3}\phi e,\text{ \ }x)\text{ }%
[\phi e,\xi ]=-a_{3}e-(\lambda -\alpha )\phi e,\text{ \ }  \label{iremmm0} \\
xi)\text{ }[e,\phi e] &=&b_{3}e-b_{4}\phi e,  \notag \\
xii)\nabla _{\xi }h &=&\xi (\lambda )s-2a_{3}h\phi ,\text{ \ \ }%
xiii)h^{2}-\alpha ^{2}\phi ^{2}=\frac{1}{2}S(\xi ,\xi )\phi ^{2}.  \notag
\end{eqnarray}%
where $a_{3}=g(\nabla _{\xi }e,\phi e),b_{3}=-\frac{1}{2\lambda }\left[
\sigma (\phi e)+(\phi e)(\lambda )\right] $ and $b_{4}=\frac{1}{2\lambda }%
\left[ \sigma (e)-e(\lambda )\right] .$
\end{lemma}

\begin{proof}
By $\nabla \xi =\alpha \phi ^{2}+\phi h$, we have $iii),vi).$

Using $\phi $-basis, we\ have%
\begin{eqnarray*}
\nabla _{\xi }\phi e &=&-g(\nabla _{\xi }\phi e,e)e+g(\nabla _{\xi }\phi
e,\phi e)\phi e+g(\nabla _{\xi }\phi e,\xi )\xi \\
&=&g(\phi e,\nabla _{\xi }e)e=a_{3}e,
\end{eqnarray*}

So we prove $viii)$ . The proofs of other covariant derivative equalities
are similar to $viii).$

Setting $X=e,$ $Y=\phi e,$ \ $Z=\xi $ in the equation (\ref{THREE DIM
CURVATURE}), we have%
\begin{equation*}
R(e,\phi e)\xi =-g(Qe,\xi )\phi e+g(Q\phi e,\xi )e.
\end{equation*}%
Since $\sigma (X)$ $=$ $g(Q\xi ,X),$ we have 
\begin{equation}
R(e,\phi e)\xi =-\sigma (e)\phi e+\sigma (\phi e)e.  \label{iremmm2}
\end{equation}%
On the other hand, by using (\ref{b3}), we have 
\begin{eqnarray}
R(e,\phi e)\xi &=&(\nabla _{e}\phi h)\phi e-(\nabla _{\phi e}\phi h)e  \notag
\\
&=&(-2b_{3}\lambda -(\phi e)(\lambda ))e+(-2b_{4}\lambda -e(\lambda ))\phi e.
\label{iremmm3}
\end{eqnarray}%
Comparing (\ref{iremmm3}) with (\ref{iremmm2}), we get%
\begin{equation*}
\sigma (e)=e(\lambda )+2b_{4}\lambda ,~\ \sigma (\phi e)=-(\phi e)(\lambda
)-2b_{3}\lambda .\ 
\end{equation*}%
Hence, the functions $b_{3}$ and $b_{4}$ are obtained from the last equation.
\end{proof}

\begin{lemma}
\label{le3}Let $(M,\phi ,\xi ,\eta ,g)$ be a $3$-dimensional almost $\alpha $%
-para-Kenmotsu manifold with $h$ of $\mathfrak{h}_{3}$\emph{\ }\textit{type}%
. Then the Ricci operator $Q$ is given by%
\begin{equation}
Q=a\text{ }I+b\eta \otimes \xi -2\alpha \phi h-\phi (\nabla _{\xi }h)+\sigma
(\phi ^{2})\otimes \xi -\sigma (e)\eta \otimes e+\sigma (\phi e)\eta \otimes
\phi e,  \label{iremmm7}
\end{equation}%
where $a$ and $b$ are smooth functions defined by $a$ $=\alpha ^{2}+\lambda
^{2}+\frac{r}{2}$ and $b=-3(\lambda ^{2}+\alpha ^{2})-\frac{r}{2},$
respectively.
\end{lemma}

\begin{proof}
Using (\ref{THREE DIM CURVATURE}), we get%
\begin{equation*}
R(X,\xi )\xi =S(\xi ,\xi )X-S(X,\xi )\xi +QX-\eta (X)Q\xi -\frac{r}{2}%
(X-\eta (X)\xi ),
\end{equation*}%
for any vector field $X.$ By (\ref{iremm}), the last equation reduces to%
\begin{equation}
QX=-\alpha ^{2}\phi ^{2}X+h^{2}X-2\alpha \phi hX-\phi (\nabla _{\xi
}h)X-S(\xi ,\xi )X+S(X,\xi )\xi +\eta (X)Q\xi +\frac{r}{2}(X-\eta (X)\xi ).
\label{QX}
\end{equation}%
By writing $S(X,\xi )=S(\phi ^{2}X,\xi )+\eta (X)S(\xi ,\xi )$ in (\ref{QX}%
), we obtain 
\begin{equation}
QX=\frac{S(\xi ,\xi )}{2}\phi ^{2}X-2\alpha \phi hX-\phi (\nabla _{\xi
}h)X-S(\xi ,\xi )X+S(\phi ^{2}X,\xi )\xi +\eta (X)S(\xi ,\xi )\xi +\eta
(X)Q\xi +\frac{r}{2}\phi ^{2}X.  \label{iremmm8}
\end{equation}%
On the other hand $S$ can be written with respect to the orthonormal basis $%
\{e,\phi e,\xi \}$ as 
\begin{equation}
Q\xi =-\sigma (e)e+\sigma (\phi e)\phi e+S(\xi ,\xi )\xi .  \label{iremmm9}
\end{equation}%
Using (\ref{iremmm9}) in (\ref{iremmm8}), we have%
\begin{eqnarray}
QX &=&\left( \alpha ^{2}+\lambda ^{2}+\frac{r}{2}\right) X+\left( -3(\lambda
^{2}+\alpha )-\frac{r}{2}\right) \eta (X)\xi -2\alpha \phi hX
\label{iremmm10} \\
&&-\phi (\nabla _{\xi }h)X+\sigma (\phi ^{2}X)\xi -\eta (X)\sigma (e)e+\eta
(X)\sigma (\phi e)\phi e,  \notag
\end{eqnarray}%
for arbitrary vector field $X$. This completes the proof.
\end{proof}

\begin{theorem}
\bigskip \label{k mu vu}Let $(M,\phi ,\xi ,\eta ,g)$ be a $3$-dimensional
almost $\alpha $-para-Kenmotsu manifold\textit{. If \ the characteristic
vector field }$\xi $\textit{\ is harmonic map then }almost $\alpha $%
-paracosymplectic $(\kappa ,\mu ,\nu )$-manifold\textit{\ always exist on
every open and dense subset of }$M.$\textit{\ Conversely, if }$M$\textit{\
is an }almost $\alpha $-paracosymplectic $(\kappa ,\mu ,\nu )$-manifold%
\textit{\ then the characteristic vector field }$\xi $\textit{\ is harmonic
map.}
\end{theorem}

\begin{proof}
We will prove theorem for three cases respect to chosen (pseudo) orthonormal
basis.

\textit{Case 1: }We assume that $h$ is $\mathfrak{h}_{1}$ type.

Since $\xi $ is a harmonic vector field, $\xi $ is an eigenvector of $Q.$
Hence we deduce that $\sigma =0.$ Putting $s=\frac{1}{\lambda }h$ in (\ref%
{irem0}) xii) , we find%
\begin{equation}
Q=(\frac{r}{2}+\alpha ^{2}-\lambda ^{2})I+(-\frac{r}{2}+3(\lambda
^{2}-\alpha ^{2}))\eta \otimes \xi -2a_{1}h-(2\alpha +\frac{\xi (\lambda )}{%
\lambda })\phi h.  \label{irem11}
\end{equation}%
Setting $Z=\xi $ in (\ref{THREE DIM CURVATURE}) and using (\ref{irem11}), we
obtain%
\begin{equation*}
R(X,Y)\xi =(-\alpha ^{2}+\lambda ^{2})(\eta (Y)X-\eta (X)Y)-2a_{1}(\eta
(Y)hX-\eta (X)hY)-(2\alpha +\frac{\xi (\lambda )}{\lambda })(\eta (Y)\phi
hX-\eta (X)\phi hY),
\end{equation*}%
where the functions $\kappa ,$ $\mu $ and $\nu $ defined by $\kappa =\frac{%
S(\xi ,\xi )}{2}=(\lambda ^{2}-\alpha ^{2}),$ $\mu =-2a_{1},$ $\nu
=-(2\alpha +\frac{\xi (\lambda )}{\lambda }),$ respectively. Moreover, using
(\ref{irem11}), we have $Q\phi -\phi Q=2\mu h\phi -2\nu h.$

\textit{Case 2: }Secondly, let $h$ be $\mathfrak{h}_{2}$ type.

Putting $\sigma =0$ in (\ref{QH2}) and using (\ref{iremm0}) $xii)$ we get%
\begin{equation}
Q=(\frac{r}{2}+\alpha ^{2})I-(\frac{r}{2}+3\alpha ^{2})\eta \otimes \xi
-2a_{2}h-2\alpha \phi hX.  \label{iremm11}
\end{equation}%
When $\xi $ $=Z$ in (\ref{THREE DIM CURVATURE}) we obtain%
\begin{eqnarray}
R(X,Y)\xi &=&-S(X,\xi )+S(Y,\xi )-\eta (X)QY  \label{iremm13} \\
&&+\eta (Y)QX+\frac{r}{2}(\eta (X)Y-\eta (Y)X),  \notag
\end{eqnarray}%
for any vector fields $X,Y.$ By applying (\ref{iremm11}) in (\ref{iremm13}),
we have%
\begin{equation*}
R(X,Y)\xi =-\alpha ^{2}(\eta (Y)X-\eta (X)Y)-2a_{2}(\eta (Y)hX-\eta
(X)hY)-2\alpha (\eta (Y)\phi hX-\eta (X)\phi hY)
\end{equation*}%
where the functions $\kappa ,$ $\mu $ and $\nu $ defined by $\kappa =\frac{%
S(\xi ,\xi )}{2}=-\alpha ^{2},$ $\mu =-2a_{2},$ $\nu =-2\alpha ,$
respectively . Furthermore, by (\ref{iremm11}), we have $Q\phi -\phi Q=2\mu
h\phi -2\nu h.$

\textit{Case 3:}\ Finally, we suppose that $h$ is $\mathfrak{h}_{3}$ type.

Since \textit{\ }$\xi $\textit{\ }is a harmonic map, we have $\sigma =0$.
Putting $s=\frac{1}{\lambda }h$ in (\ref{iremmm7}) we get%
\begin{equation}
Q=a\text{ }I+b\eta \otimes \xi -2\alpha \phi h-\phi (\nabla _{\xi }h),
\label{iremmm11}
\end{equation}%
Setting $\xi =$ $Z$ in (\ref{THREE DIM CURVATURE}) we again obtain%
\begin{eqnarray}
R(X,Y)\xi &=&-S(X,\xi )+S(Y,\xi )-\eta (X)QY  \label{iremmm13} \\
&&+\eta (Y)QX+\frac{r}{2}(\eta (X)Y-\eta (Y)X),  \notag
\end{eqnarray}%
for any vector fields $X,Y.$ Using (\ref{iremmm11}) in (\ref{iremmm13}), we
get%
\begin{equation*}
R(X,Y)\xi =-(\alpha ^{2}+\lambda ^{2})(\eta (Y)X-\eta (X)Y)-2a_{3}(\eta
(Y)hX-\eta (X)hY)-(2\alpha +\frac{\xi (\lambda )}{\lambda })(\eta (Y)\phi
hX-\eta (X)\phi hY),
\end{equation*}%
where the functions $\kappa ,$ $\mu $ and $\nu $ are defined by $\kappa
=-(\alpha ^{2}+\lambda ^{2}),$ $\mu =-2a_{3},$ $\nu =-(2\alpha +\frac{\xi
(\lambda )}{\lambda })$, respectively. By help of (\ref{iremmm11}), we get $%
Q\phi -\phi Q=2\mu h\phi -2\nu h.$

This completes the proof.
\end{proof}

\section{Example}

\begin{example}
\label{E}We consider the $3$-dimensional manifold%
\begin{equation*}
M=\left\{ (x,y,z)\in R^{3}\mid x\neq 0,\text{ }y\neq 0\right\}
\end{equation*}%
and the vector fields%
\begin{equation*}
e_{1}=\frac{\partial }{\partial x},\ \ e_{2}=\phi e_{1}=\frac{\partial }{%
\partial y},\ \ e_{3}=\xi =x\ \frac{\partial }{\partial x}+(y+2x)\frac{%
\partial }{\partial y}+\frac{\partial }{\partial z}\ .
\end{equation*}%
The $1$-form $\eta =dz$ $\ $and the fundamental $2$-form $\Phi =dx\wedge
dy-(y+2x)dx\wedge dz-xdy\wedge dz$ defines an almost para-Kenmotsu manifold.

Let $g,$ $\phi $ be the pseudo-Riemannian metric and the $(1,1)$-tensor
field given by%
\begin{eqnarray*}
g &=&\left( 
\begin{array}{ccc}
1 & 0 & -x \\ 
0 & -1 & y+2x \\ 
-x & y+2x & 1-3x^{2}-4xy-y^{2}%
\end{array}%
\right) , \\
\phi &=&\left( 
\begin{array}{ccc}
0 & 1 & -(y+2x) \\ 
1 & 0 & -x \\ 
0 & 0 & 0%
\end{array}%
\right) .
\end{eqnarray*}%
We easily get 
\begin{eqnarray*}
\left[ e_{1},e_{2}\right] &=&0, \\
\left[ e_{1},e_{3}\right] &=&e_{1}+2e_{2}, \\
\left[ e_{2},e_{3}\right] &=&e_{2}.
\end{eqnarray*}%
Moreover, the above example is an almost para-Kenmotsu $(\kappa ,\mu ,\nu
)=(1,1,-2)$-space.
\end{example}


\begin{thebibliography}{99}
\bibitem{ACP} MTK Abbasi, G. Calvaruso, D. Perrone, \emph{Harmonicity of
unit vector fields with respect to Riemannian }$g$\emph{-natural metrics, }%
Different. Geom. Appl. \textbf{27} (2009), 157-169.

\bibitem{Bejan} C.-L.~Bejan, \emph{Almost parahermitian structures on the
tangent bundle of an almost para-coHermitian manifold}, In: The Proceedings
of the Fifth National Seminar of Finsler and Lagrange Spaces (Bra\thinspace
sov, 1988), 105--109, Soc. \TEXTsymbol{\vert},Stiin\thinspace te Mat. R. S.
Romania, Bucharest, 1989.

\bibitem{Blair} D.~E.~Blair, \emph{Riemannian geometry of contact and
symplectic manifolds}, Progress Math. Vol \textbf{203}, Birkh\"auser,
Boston, MA, 2010.

\bibitem{Boeckx} E.~Boeckx, \emph{A full classification of contact metric $%
(\kappa,\mu)$-spaces}, Illinois J. Math. \textbf{44} (2000), no 1., 212-219.

\bibitem{BuchRosc} K.~Buchner, R.~Ro\,sca, \emph{Vari\'etes
para-coK\"ahlerian \'a champ concirculaire horizontale}, C. R. Acad. Sci.
Paris \textbf{285} (1977), Ser. A, 723--726.

\bibitem{BuchRosc2} K.~Buchner, R.~Ro\,sca, \emph{Co-isotropic submanifolds
of a para-coK\"ahlerian manifold with concicular vector field}, J. Geometry 
\textbf{25} (1985), 164--177.

\bibitem{calvaruso1} G. Calvaruso, \emph{Harmonicity properties of invariant
vector fields on three-dimensional Lorentzian Lie groups}, J. Geom. Phys.%
\textbf{\ 61} (2011), 498-515.

\bibitem{Biz} B. Cappelletti-Montano, I. Kupeli Erken, C. Murathan, \emph{%
Nullity conditions in paracontact geometry}, Diff. Geom. Appl. \textbf{30}
(2012), 665--693.

\bibitem{CapMonNicYud} B.~Cappelletti-Montano, A. D. Nicola, I. Yudin, \emph{%
A survey on cosymplectic geometry}, Rev. Math. Phys. \textbf{25}, no. 3
(2013),

\bibitem{CapDrag} M.~Capursi, S.~Dragomir, \emph{On manifolds admitting
metrics which are locally conformal to cosymplectic metrics: their canonical
foliations, Boothby -Wang fiberings and homology type}, Colloq. Math. 
\textbf{64} (1993), 29--40.

\bibitem{CarMol} A.~Carriazo, V.~Martin-Molina, \emph{Almost cosymplectic
and almost Kenmotsu $(\kappa,\mu,\nu)$-spaces}, Medit. J. Math. \textbf{10}
(2013), no 3., 1551--1571 (English).

\bibitem{ChinMar} D.~Chinea, J.~C.~Marrero, \emph{Conformal changes of
almost cosymplectic manifolds}, Rend. Mat. Appl. \textbf{12} (1992),
849--867.

\bibitem{CFG} V.~Cruceanu, P.~Fortuny, P.~M.~Gadea, \emph{A survey on
paracomplex geometry}, Rocky Mount. J. Math. \textbf{26} (1996), 83--115.

\bibitem{DACKO} P. Dacko, \emph{On almost para-cosymplectic manifolds},
Tsukuba J. Math. \textbf{28} (2004), 193--213.

\bibitem{Dacko15} P. Dacko, \emph{Almost para-cosymplectic manifolds with
contact Ricci potential}. Available in Arxiv: 1308.6429v1 $\left[ \text{%
math. DG}\right] $

\bibitem{DacOl1} P.~Dacko, Z.~Olszak, \emph{On almost cosymplectic $%
(\kappa,\mu,\nu)$-spaces}, PDEs, submanifolds and affine differential
geometry, 211--220, Banach Center Publ., \textbf{69}, Polish Acad. Sci.,
Warsaw 2005.

\bibitem{DACKO2} P. Dacko, Z. Olszak, \emph{On weakly para-cosymplectic
manifolds of dimension 3}, J. Geom. Phys.\textbf{\ 57} (2007), 561-570.

\bibitem{Dil-1} G.~Dileo, \emph{A classification of certain almost $\alpha $%
-Kenmotsu manifolds}, Kodai Math. J. \textbf{34}, no. 3 (2011), 426--445.

\bibitem{DILEO} G. Dileo, A.M. Pastore, \emph{Almost Kenmotsu manifolds and
local symmetry}, Bull. Belg. Math. Soc. Simon Stevin \textbf{14} (2007),
343--354.

\bibitem{DilPast2} G.~Dileo, A.~M. Pastore, \emph{Almost Kenmotsu manifolds
and nullity distributions}, J. Geom. \textbf{93} (2009), 46--61.

\bibitem{Fal-1} M.~Falcitelli, \emph{A class of almost contact metric
manifolds with pointwise constant $\phi$-sectional curvature}, Math. Balk.
N.S. \textbf{22} (2008), 133-153.

\bibitem{Falcitelli} M.~Falcitelli, \emph{Curvature of locally conformal
cosymplectic manifolds}, Publ. Math. Debrecen, \textbf{72/3-4}
(2008),385--406.

\bibitem{kaneyuki1} S. Kaneyuki , FL.Williams, \emph{Almost paracontact and
parahodge structures on manifolds}, Nagoya Math. J 1985; 99: 173--187.

\bibitem{KimPak} T.~W.~Kim, H.~K.~Pak, \emph{Canonical foliations of certain
classes of almost contact metric structures}, Acta Math. Sin. (Engl. Ser.) 
\textbf{21} no. 4(2005) 841--846.

\bibitem{kupmur} I. Kupeli Erken, C. Murathan, \emph{A Complete Study of
Three-Dimensional Paracontact (}$\kappa ,\mu ,\nu )$-\emph{spaces},
Submitted. Available in Arxiv:1305.1511 $\left[ \text{math. DG}\right] $

\bibitem{MAGID} M. Magid, \emph{Lorentzian isoparametric hypersurfaces},
Pacific J. Math. \textbf{118} (1985) MR 0783023,165--198.

\bibitem{marrero} J.~C.~Marrero, \emph{Locally conformal cosymplectic
manifolds foliated by generalized Hopf manifolds}, Rend. di Mat. Serie VII, 
\textbf{12}, (1992), 305--327.

\bibitem{MaMol} V.~Marti\'n-Molina, \emph{Paracontact metric manifolds
without a contact metric counterpart}, Available in ArxiV: 1312.6518 [math.
DG].

\bibitem{Olszak} Z.~Olszak, \emph{Locally conformal almost cosymplectic
manifolds}, Colloq. Math.\textbf{57} (1989), 73--87.

\bibitem{OlK} Z.~Olszak, \emph{Almost cosymplectic leaves with K\"ahlerian
leaves}, Tensor N.S. \textbf{46} (1987), 117--124.

\bibitem{ONEILL} B. O'Neill, \emph{Semi-Riemann Geometry}, Academic Press.
New York, 1983.

\bibitem{HAKAN} H.\"{O}zt\"{u}rk, N. Aktan, C. Murathan, \emph{Almost }$%
\alpha $-\emph{cosymplectic (}$\kappa ,\mu ,\nu )$-\emph{spaces}, Submitted.
Available in Arxiv:1007.0527v1 $\left[ \text{math. DG}\right] $.

\bibitem{PastSal} A.~M.~Pastore, V.~Saltarelli, \emph{Generalized nullity
distributions on almost Kenmotsu manifolds}, Int. Electron. J. Geom. \textbf{%
4} (2011), no 2, 168--183.

\bibitem{PERRONE} D. Perrone, \emph{Minimal reeb vector fields on almost
cosymplectic manifolds, }Kodai Math. J. \textbf{36} (2013), 258-274.

\bibitem{a.z.petrov} AZ. Petrov, \emph{Einstein spaces}, Pergamon Press,
Oxford. MR 0244912, 1969.

\bibitem{RoscVanh} R.~Ros\,sca, L. Vanhecke, \emph{S\'ur une vari\'et\'e
presque paracok\"ahl\'erienne munie d'une connexion self-orthogonale
involutive}, Ann. \,Sti. Univ. ``Al. I. Cuza'' Ia\,si \textbf{22} (1976),
49--58.

\bibitem{Sal} V.~Saltarelli, \emph{Three dimensional almost Kenmotsu
manifolds satisfying certain nullity conditions},Available in Arxiv:
1007.1443v4 $\left[ \text{math. DG}\right] $

\bibitem{Tanno} S.~Tanno, \emph{Partially conformal transformations with
respect to (m-1)-dimensional distributionsof m-dimensional Riemannian
manifolds}, Tohoku Math. J. (2) \textbf{17}, no. 4 (1965), 358--409.

\bibitem{Welyczko} J.~We\l {}yczko, \emph{On basic curvature identities for
almost (para)contact metric manifolds}. Available in Arxiv: 1209.4731v1 $%
\left[ \text{math. DG} \right] $.

\bibitem{Za} S. Zamkovoy, \emph{Canonical connections on paracontact
manifolds}, Ann. Glob. Anal. Geom. \textbf{36} (2009), 37--60.
\end{thebibliography}
\end{document}